\documentclass[review,hidelinks,onefignum,onetabnum]{siamart220329}
\nolinenumbers
\usepackage{lipsum}
\usepackage{amsfonts}
\usepackage{graphicx}
\usepackage{epstopdf}
\usepackage{algorithmic}
\usepackage{cleveref}
\usepackage{tikz-cd}
\usepackage{amsmath}
\usepackage{subcaption}
\usepackage{mathtools}
\usepackage{todonotes}
\usepackage{pgfplots}
\usepackage{bm}
\usetikzlibrary{spy, pgfplots.groupplots}
\usepackage[sort]{cite}
\usepackage[final]{changes}

\definecolor{seaborngreen}{rgb}{0.3333333333333333, 0.6588235294117647, 0.40784313725490196}  % #55A868
\definecolor{seaborncyan}{rgb}{0.39215686274509803, 0.7098039215686275, 0.803921568627451}  % #64B5CD
\definecolor{seabornblue}{rgb}{0.2980392156862745, 0.4470588235294118, 0.6901960784313725}  % #4C72B0
\definecolor{seabornpurple}{rgb}{0.5058823529411764, 0.4470588235294118, 0.6980392156862745}  % #8172B2
\definecolor{seabornred}{rgb}{0.7686274509803922, 0.3058823529411765, 0.3215686274509804}  % #C44E52
\definecolor{seabornorange}{rgb}{0.958, 0.476, 0.206}  % #F47935
\definecolor{seabornsand}{rgb}{0.8, 0.7254901960784313, 0.4549019607843137}  % #CCB974

\newcommand{\bbP}{\mathbb{P}}
\newcommand{\bbQ}{\mathbb{Q}}
\newcommand{\bbR}{\mathbb{R}}

\newcommand{\bbU}{\mathbb{U}}
\newcommand{\bbV}{\mathbb{V}}

\newcommand{\bbX}{\mathbb{X}}

\newcommand{\bfQ}{\mathbf{Q}}

\newcommand{\bfe}{\mathbf{e}}

\newcommand{\bfr}{\mathbf{r}}

\newcommand{\bfu}{\mathbf{u}}
\newcommand{\bfv}{\mathbf{v}}
\newcommand{\bfw}{\mathbf{w}}

\newcommand{\bfzero}{\mathbf{0}}

\newcommand{\bfmu}{\bm{\mu}}
\newcommand{\bfomega}{\bm{\omega}}
\newcommand{\bfvarphi}{\bm{\varphi}}
\newcommand{\bftheta}{\bm{\theta}}

\newcommand{\rmd}{\mathrm{d}}

\newcommand{\calI}{\mathcal{I}}

\DeclareMathOperator{\curl}{curl}
\let\div\relax
\DeclareMathOperator{\div}{div}

\let\Re\relax
\newcommand{\Re}{\mathrm{Re}}
\let\Pr\relax
\newcommand{\Pr}{\mathrm{Pr}}

\newcommand{\vertiii}[1]{{\left\vert\kern-0.25ex\left\vert\kern-0.25ex\left\vert #1
    \right\vert\kern-0.25ex\right\vert\kern-0.25ex\right\vert}}

\ifpdf
  \DeclareGraphicsExtensions{.eps,.pdf,.png,.jpg}
\else
  \DeclareGraphicsExtensions{.eps}
\fi

\newsiamremark{remark}{Remark}
\newsiamremark{hypothesis}{Hypothesis}
\newsiamthm{problem}{Problem}
\crefname{hypothesis}{Hypothesis}{Hypotheses}

\newsiamthm{claim}{Claim}

\headers{Conservative and dissipative numerical integrators}{B.~D.~Andrews and P.~E.~Farrell}

\title{Enforcing conservation laws and dissipation inequalities numerically via auxiliary variables\thanks{Submitted to the editors DATE.
\funding{This work was funded by the Engineering and Physical Sciences Research Council
[grant numbers EP/R029423/1 and EP/W026163/1], the EPSRC Energy Programme [grant number
EP/W006839/1], a CASE award from the UK Atomic Energy Authority, the Donatio Universitatis
Carolinae Chair ``Mathematical modelling of multicomponent systems'', and the UKRI Digital Research Infrastructure Programme through the Science and Technology Facilities Council's Computational Science Centre for Research Communities (CoSeC).
For the purpose of open access, the authors have applied a CC BY public copyright licence to any author accepted manuscript (AAM) arising from this submission. No new data were generated or analyzed during this work.}
}
}

\author{Boris D.~Andrews\thanks{Mathematical Institute, University of Oxford, UK
  (\email{boris.andrews@maths.ox.ac.uk}).}
\and Patrick E.~Farrell\thanks{Mathematical Institute, University of Oxford, UK
     and
     Mathematical Institute, Faculty of Mathematics and Physics, Charles University, Czechia
  (\email{patrick.farrell@maths.ox.ac.uk}).}
}

\usepackage{amsopn}

\ifpdf
\hypersetup{
  pdftitle={Enforcing conservation laws and dissipation inequalities numerically via auxiliary variables},
  pdfauthor={B.~D.~Andrews and P.~E.~Farrell}
}
\fi

\begin{document}

\maketitle

% REQUIRED
\begin{abstract}
We propose a general strategy for enforcing multiple conservation laws and dissipation inequalities in the numerical solution of initial value problems. The key idea is to represent each conservation law or dissipation inequality by means of an associated test function; we introduce auxiliary variables representing the projection of these test functions onto a discrete test set, and modify the equation to use these new variables. We demonstrate these ideas by their application to the Navier--Stokes equations. We generalize to arbitrary order the energy-dissipating and helicity-tracking scheme of Rebholz for the incompressible Navier--Stokes equations, and devise \added{a} \deleted{the first} time discretization of the compressible equations that conserves mass, momentum, and energy, and provably dissipates entropy.
\end{abstract}

% REQUIRED

% REQUIRED

\section{Introduction}

Structure-preserving discretizations in time for differential equation initial value problems (geometric numerical integrators) have been intensively studied because of their fundamental importance for accurate numerical simulation \cite{deVogelaere_1956, SanzSerna_Calvo_1994, Budd_Piggott_2003, Hairer_et_al_2006, Hairer_Lubich_Wanner_2006, Christiansen_MuntheKaas_Owren_2011, Blanes_Casas_2017, Iserles_Quispel_2018}.
In this work we develop a new approach for devising arbitrary-order discretizations that conserve multiple arbitrary invariants and accurately reproduce dissipation structures.

The key novelty of our approach is the systematic introduction of auxiliary variables to modify existing schemes.
We will show that each conservation law or dissipation inequality to be preserved corresponds to an associated test function;
the auxiliary variables introduced are certain projections of these associated test functions onto a discrete test set.
Our approach gives conservative and accurately dissipative schemes for previously inaccessible problems, while generalizing several distinct ideas from the literature, including the methods proposed in \cite{Simo_Armero_1994, Rebholz_2007, Hu_Lee_Xu_2021, Laakmann_Hu_Farrell_2023} and the frameworks proposed in \cite{McLachlan_Quispel_Robidoux_1999, Betsch_Steinmann_2000a, Cohen_Hairer_2011, Hairer_Lubich_2014, Egger_Habrich_Shashkov_2021, Giesselmann_Karsai_Tscherpel_2025}.

We present our approach in Section~\ref{sec:incompressible}, employing as a didactic example the incompressible Navier--Stokes equations in three dimensions.
This section derives an arbitrary-order generalization of the energy-dissipating and helicity-tracking discretization proposed by Rebholz \cite{Rebholz_2007}.
We then employ the same ideas in Section~\ref{sec:compressible} to the more complex context of the compressible Navier--Stokes equations, deriving \added{a} \deleted{what we believe to be the first} time discretization of these equations that conserves mass, momentum, and energy, while also dissipating entropy\added{, without needing to reparametrise the system to make these quantities of interest linear}.
We discuss related literature and the relationships to our proposed framework in Section~\ref{sec:related_literature}.
We conclude in Section~\ref{sec:conclusions}.

\section{General strategy: incompressible Navier--Stokes}\label{sec:incompressible}

We now present the general framework.

%\paragraph{Example (incompressible Navier--Stokes)}
To fix ideas, throughout this section we will employ the incompressible Navier--Stokes equations as our running example.
The equations can be written in strong form as
\begin{equation}\label{eq:strongincompns}
    \dot{\bfu}  =  \bfu \times \curl\bfu - \nabla p - \frac{1}{\Re}\Delta\bfu,  \quad
             0  =  \div\bfu,
\end{equation}
over a bounded Lipschitz domain $\Omega \subset \bbR^3$.
Here $\bfu: \bbR_+ \times \Omega \to \bbR^3$ is the velocity, $p: \bbR_+ \times \Omega \to \bbR$ is the total (or Bernoulli) pressure, and $\Re > 0$ is the Reynolds number; we use a rotational (or Lamb) form for the convective term.
We consider periodic boundary conditions (BCs), with the additional constraint on the initial condition
\begin{equation}\label{eq:zero_motion_ic}
    \int_\Omega \bfu(0) = \bfzero.
\end{equation}
If \eqref{eq:zero_motion_ic} holds then any solution to \eqref{eq:strongincompns} satisfies $\int_\Omega \bfu = \bfzero$ at all times.
This condition is included to ensure certain energy estimates exist on solutions to the scheme, as required for its analysis;
in particular, we require that $\|\nabla\bfu\|$ defines a norm on $\bfu$, where $\|\cdot\|$ denotes the $L^2(\Omega)$ norm.

Our quantities of interest are the (kinetic) energy and helicity
\begin{equation}\label{eq:incomp_quantities_of_interest}
    Q_1(\bfu)  \coloneqq  \frac{1}{2}\|\bfu\|^2,  \quad
    Q_2(\bfu)  \coloneqq  \frac{1}{2}(\bfu, \curl\bfu),
\end{equation}
respectively.
Here $(\cdot, \cdot)$ denotes the $L^2(\Omega)$ inner product.
Under periodic BCs, $Q_1$ and $Q_2$ are conserved in solutions of the formal ideal limit $\Re = \infty$, while $Q_1$ is necessarily dissipated for finite $\Re$; we wish to construct a timestepping scheme that preserves these behaviors.

We proceed in six stages.

\subsection{Definition of semi-discrete form}\label{sec:general_framework_a}
We define first an abstract semi-discrete formulation of a transient PDE, discretized in space only. This is posed over a general affine space
\begin{equation}\label{eq:affinespace}
    \bbX \coloneqq \left\{u \in C^1(\bbR_+; \bbU) : u(0) \text{ satisfies known initial data}\right\}.
\end{equation}
Here, $\bbU$ denotes an appropriate finite-dimensional spatial function space\footnote{Since $\bbU$ must be a function space, not an affine space, this framework is in general only applicable to problems with homogeneous or periodic BCs.}.
The abstract semi-discrete weak problem is then as follows: find $u \in \bbX$ such that
\begin{equation}\label{eq:generalweak}
    M(u; \dot{u}, v) = F(u; v)
\end{equation}
at all times $t \in \bbR_+$ and for all $v \in \bbU$.
Here $M : \bbU \times \bbU \times \bbU \to \bbR$ is possibly nonlinear in $u$, but linear in $\dot{u}$ and $v$; this is the significance of the semicolon.
Similarly, $F : \bbU \times \bbU \to \bbR$ is possibly nonlinear in $u$, but linear in the test function $v$.

%\paragraph{Example (incompressible Navier--Stokes)}
Turning to our running example~\eqref{eq:strongincompns} we construct a vanilla semi-discretization, which we show can be written in the form \eqref{eq:generalweak}.
Let $\bbV \subset H^1_{\text{per}}(\Omega; \mathbb{R}^3)$, $H^1$-vector fields with periodic boundary conditions, and $\bbQ \subset L^2_0(\Omega)$ be suitable inf-sup stable \cite[Chap.~49]{Ern_Guermond_2021b} finite-dimensional function spaces.
The first formulation is: find $(\bfu, p) \in C^1(\bbR_+; \bbV) \times C^0(\bbR_+; \bbQ)$, satisfying known initial conditions in $\bfu$ such that
\begin{equation}\label{eq:incompnsweak_lagrange}
    (\dot{\bfu}, \bfv)  =  (\bfu \times \curl\bfu, \bfv) + (p, \div\bfv) - \frac{1}{\Re}(\nabla\bfu, \nabla\bfv),  \quad
                     0  =  (\div\bfu, q),
\end{equation}
at all times $t \in \bbR_+$ and for all $(\bfv, q) \in \bbV \times \bbQ$.
In its current form, there are no time derivatives on the pressure term, $p$, implying \eqref{eq:incompnsweak_lagrange} can not be written in the form of \eqref{eq:generalweak}; specifically, \eqref{eq:incompnsweak_lagrange} represents a differential-algebraic system \cite{Ascher_Petzold_1998}.
To remedy this, define the discretely divergence-free subspace $\bbU \subset \bbV$,
\begin{equation}\label{eq:u_space_incompns}
    \bbU \coloneqq \left\{\bfu \in \bbV : (\div\bfu, q) = 0 \text{ for all } q \in \bbQ \text{ and } \int_\Omega \bfu = \bfzero\right\}.
\end{equation}
We can effectively eliminate both $p$ and the mass-conservation equation by posing the semi-discretization in $\bbU$, while further incorporating the condition $\int_\Omega \bfu = \bfzero$;
$\bbX$ is then defined from $\bbU$ as in \eqref{eq:affinespace}.
The semi-discretization then states: find $\bfu \in \bbX$ such that
\begin{equation}\label{eq:incompnsweak}
    (\dot{\bfu}, \bfv)  =  (\bfu \times \curl\bfu, \bfv) - \frac{1}{\Re}(\nabla\bfu, \nabla\bfv)
\end{equation}
at all times $t \in \bbR_+$ and for all $\bfv \in \bbU$\footnote{In practice, after discretizing in time the implementation would typically enforce the constraints in $\bbU$ through Lagrange multipliers, more closely resembling \eqref{eq:incompnsweak_lagrange}. However, the formulation in terms of discretely divergence-free $\bbU$ simplifies the construction of the scheme.}. This is in the form of \eqref{eq:generalweak} with $M$, $F$ given by
\begin{equation}
    M(\bfu; \dot{\bfu}, \bfv)  \coloneqq  (\dot{\bfu}, \bfv),  \quad
    F(\bfu; \bfv)              \coloneqq  (\bfu \times \curl\bfu, \bfv) - \frac{1}{\Re}(\nabla\bfu, \nabla\bfv).
\end{equation}

\subsection{Definition of timestepping scheme}\label{sec:general_framework_b}
To fully discretize our problem, we define $\bbX_n$ on a timestep $T_n = [t_n, t_{n+1}]$.
We employ polynomials in time of degree $S \ge 1$:
\begin{equation}\label{eq:solution_space}
    \bbX_n  \coloneqq  \!\left\{u \in \bbP_S(T_n; \bbU) : u(t_n) \text{ satisfies known initial data}\right\}\!.
\end{equation}
The abstract timestepping scheme is then as follows: find $u \in \bbX_n$ such that
\begin{equation}\label{eq:generalweak_time}
    \calI_n[M(u; \dot{u}, v)]  =  \calI_n[F(u; v)]
\end{equation}
for all $v \in \dot{\bbX}_n = \bbP_{S-1}(T_n; \bbU)$, where $\calI_n$ is a chosen linear operator that approximates the integral over $T_n$ (a quadrature rule):
\begin{equation}
    \calI_n[\phi] \approx \int_{T_n} \phi.
\end{equation}
The approximation must be sign-preserving, i.e.
\begin{subequations}
\begin{equation}\label{eq:Isignpreserving}
    \phi \ge 0 \implies \calI_n[\phi] \ge 0,
\end{equation}
appropriately scaled in $\Delta t_n = t_{n+1} - t_n$, i.e.
\begin{equation}\label{eq:Iscaling}
    \calI_n[1] = \Delta t_n,
\end{equation}
\end{subequations}
and the map $\phi \mapsto \calI_n[\phi^2]^\frac{1}{2}$ must define a norm on $\bbP_{S-1}(T_n)$ to avoid degeneracy in the time discretization.\footnote{
    \added{In the case where $M(u; \cdot, \cdot)$ is independent of $u$ and defines an inner product on $\bbU$, this condition implies $\calI_n[M(\cdot, \cdot)]$ defines an inner product on $\dot{\bbX}_n$;
    in the trivial case $F = 0$, it is thus necessary and sufficient to ensure the discretization is well-posed.
    While a proof is not presented here, it is reasonable to expect that it remains a necessary condition for well-posedness in the general case.
    The use of $\calI_n$ for which $\phi \mapsto \calI_n[\phi^2]^\frac{1}{2}$ does not define a norm on $\bbP_{S-1}(T_n)$ would be comparable to proposing a collocation method with fewer collocation points than the polynomial degree.}
}
Examples of such linear operators include the exact integral, and any $S$-stage quadrature rule with positive weights.

%\paragraph{Example (incompressible Navier--Stokes)}
No specific choice of $\calI_n$ is required for the incompressible Navier--Stokes example.
For concreteness, we might choose $\calI_n$ to be a Gauss--Legendre quadrature, yielding a Gauss collocation method as our base timestepping scheme.

\subsection{Identification of associated test functions}\label{sec:general_framework_c}
The properties we wish to preserve (conservation laws or dissipation structures) are associated with particular choices of test functions.
For Fr\'echet-differentiable quantities of interest $(Q_p : \bbU \to \bbR)_{p=1}^P$, we assume there exist test functions $(w_p(u))_{p=1}^P$ such that the Fr\'echet derivatives $Q_p'(u; v) = M(u; v, w_p(u))$ for general $u, v$.
Consequently, for $u$ an exact solution of the PDE,
\begin{equation}\label{eq:general_preservation}
      Q_p(u(t_{n+1})) - Q_p(u(t_n))
    = \int_{T_n}Q_p'(u; \dot{u})
    = \int_{T_n}M(u; \dot{u}, w_p(u))
    = \int_{T_n}F(u; w_p(u)).
\end{equation}
\added{For each $p$, the behaviour of $Q_p$ is then encoded in the sign of $F(u; w_p(u))$:
for conserved $Q_p$, $F(u; w_p(u)) = 0$;
for dissipated $Q_p$, $F(u; w_p(u)) \le 0$.}

\added{This definition of $(w_p(u))$ is motivated by the general semi-discrete \eqref{eq:generalweak} and discrete \eqref{eq:generalweak_time} schemes.
These functions on $u$ are precisely those that, for each $p$, were we able to consider $v = w_p(u)$ in \eqref{eq:generalweak} (i.e.~if $w_p(u)$ were necessarily contained in $\bbU$), the desired conservation or dissipation structures on $Q_p(u)$ would necessarily hold exactly at the semi-discrete level without modification;
moreover, were we able to consider $v = w_p(u)$ in \eqref{eq:generalweak_time} (i.e.~if $w_p(u)$ were necessarily contained in $\dot{\bbX}_n$) and $\calI_n$ were the exact integral, these structures would hold on the fully discrete level also.
No constraints, however, are posed here on the space containing $w_p(u)$.
It is not generally true that $w_p(u)$ is contained in either $\bbU$ or $\dot{\bbX}_n$.}

\deleted{Note, no constraints are posed here on the space containing $w_p(u)$;
it is not generally true that $w_p(u) \in \bbU$.
For each $p$, the behaviour of $Q_p$ is then encoded in the sign of $F(u; w_p(u))$;
in particular for conserved $Q_p$, $F(u; w_p(u)) = 0$, whereas for dissipated $Q_p$, $F(u; w_p(u)) \le 0$.}

%\paragraph{Example (incompressible Navier--Stokes)}
For the incompressible Navier--Stokes equations,
the Fr\'echet derivatives of our quantities of interest~\eqref{eq:incomp_quantities_of_interest} are
\begin{equation}
Q_1'(\bfu; \bfv) = (\bfu, \bfv), \quad Q_2'(\bfu; \bfv) = (\curl \bfu, \bfv),
\end{equation}
and with $M$ the $L_2$ inner product, we conclude that the associated test functions are $\bfw_1(\bfu) = \bfu$ and $\bfw_2(\bfu) = \curl \bfu$ respectively.

Testing the equations with $\bfu \, (= \bfw_1(\bfu))$ yields the update formula for the kinetic energy $Q_1$:
\begin{subequations}\label{eq:incompnske}
\begin{align}
        Q_1(\bfu(t_{n+1})) - Q_1(\bfu(t_n))
    &=  \int_{T_n}(\bfu, \dot{\bfu})  \\
    &=  \int_{T_n}\!\left[(\bfu \times \curl\bfu, \bfu) - \frac{1}{\Re}(\nabla\bfu, \nabla\bfu)\right]\! \\
    &=  - \frac{1}{\Re}\int_{T_n}\|\nabla\bfu\|^2 \le 0.
\end{align}
\end{subequations}
\added{Since $\bfu \, (= \bfw_1(\bfu)) \in \bbU$, this dissipation structure necessarily holds for the semi-discrete form \eqref{eq:incompnsweak} without modification.}
Similarly, testing the equations with $\curl\bfu \, (= \bfw_2(\bfu))$ yields the update formula for the helicity $Q_2$:
\begin{subequations}\label{eq:incompnshel}
\begin{align}
        Q_2(\bfu(t_{n+1})) - Q_2(\bfu(t_n))
    &=  \int_{T_n}(\curl\bfu, \dot{\bfu})  \\
    &=  \int_{T_n}\!\left[(\bfu \times \curl\bfu, \curl\bfu) - \frac{1}{\Re}(\nabla\bfu, \nabla \curl\bfu)\right]\! \\
    &=  - \frac{1}{\Re}\int_{T_n}(\nabla\bfu, \nabla\curl\bfu).
\end{align}
\end{subequations}
\added{Since $\curl\bfu \, (= \bfw_2(\bfu))$ is not necessarily contained in $\bbU$, this behaviour is not necessarily preserved at the semi-discrete level.}
In each case we apply integration by parts using the periodic BCs in $\bfu$, while the advection terms vanish due to properties of the cross product.
We note that $Q_1$ is dissipated for finite $\Re$ while $Q_2$ can increase or decrease, and that both $Q_1$ and $Q_2$ are conserved in the ideal limit $\Re = \infty$.

\subsection{Introduction of auxiliary variables}\label{sec:general_framework_d}
Our aim is to replicate the conservation/dissipation properties \eqref{eq:general_preservation} discretely.
However, as it stands this cannot be done, as in general $w_p(u) \not\in \dot{\bbX}_n$;
they are not valid choices of discrete test functions.
We therefore introduce auxiliary variables $(\tilde{w}_p)_{p=1}^P$ into the formulation, computing approximations to the associated test functions $(w_p(u))$ within the discrete test space $\dot{\bbX}_n$. Namely, for all $p = 1, \dots, P$, $\tilde{w}_p \in \dot{\bbX}_n$ is defined weakly such that
\begin{equation}\label{eq:av_definition}
    \calI_n[M(u; v_p, \tilde{w}_p)] = \int_{T_n}Q'_p(u; v_p) \quad \left(= \int_{T_n}M(u; v_p, w_p(u))\right)\!,
\end{equation}
for all $v_p \in \dot{\bbX}_n$.

\begin{remark}
    \added{The right-hand side $\int_{T_n}Q'_p(u; v_p)$ of \eqref{eq:av_definition} employs exact integrals in time $\int_{T_n}\!$. As these integrals reflect those recovered when applying the fundamental theorem of calculus to each quantity of interest $Q_p$ over the timestep $T_n$ (outlined in Theorem~\ref{th:avcpg_sp} below), they may not in general be substituted for $\calI_n$, nor any other quadrature rule, without breaking the structure-preserving properties of the discretization. Schemes derived from our framework are thus necessarily implicit in time, since these exact integrals depend on the final value $u(t_{n+1})$ of $u$ over $T_n$.}

    \added{Note, it is often not possible to evaluate this integral exactly, in particular when the integrand is non-polynomial (see e.g.~the compressible Navier--Stokes example considered in Section~\ref{sec:compressible}). In such cases, we employ a quadrature rule in time of sufficiently high degree for this term that the quadrature error is on the order of machine precision.}
\end{remark}

%\paragraph{Example (incompressible Navier--Stokes)}
For the incompressible Navier--Stokes system, $(\tilde{\bfu}, \tilde{\bfomega}) \, (= (\tilde{\bfw}_1, \tilde{\bfw}_2)) \in (\dot{\bbX}_n)^2$ are defined weakly such that
\begin{align}\label{eq:incompnsavs}
    \calI_n[(\bfv_1, \tilde{\bfu})]    =  \int_{T_n}(\bfu, \bfv_1),  \quad
    \calI_n[(\bfv_2, \tilde{\bfomega})]  =  \int_{T_n}(\curl\bfu, \bfv_2),
\end{align}
for all $(\bfv_1, \bfv_2) \in (\dot{\bbX}_n)^2$.
In particular, whereas $\bfu$ is an approximation of the velocity that is continuous across time intervals of polynomial degree $S$, $\tilde{\bfu}$ is another approximation of the velocity with the same spatial discretization that is discontinuous across time intervals of polynomial degree $S-1$.\footnote{
    \added{Note, in the continuous case $\curl\bfu$ should both be divergence-free by $\div\curl = 0$, and satisfy the constraint $\int_\Omega \curl\bfu = \bfzero$ by Stokes' theorem.
    These results are continuous analogues of the restrictions on $\bbU$ in \eqref{eq:u_space_incompns};
    as such, it is appropriate to approximate $\curl \bfu$ by $\tilde{\bfomega} \in \dot{\bbX}_n = \bbP_{S-1}(T_n; \bbU)$.}
}

\begin{remark}
    In some cases, certain auxiliary variables $\tilde{w}_p$ can be computed explicitly, and are therefore not needed in the implementation. In our running example, this is the case for $\tilde{\bfu} \, (= \tilde{\bfw}_1)$.
    %We discuss this further in Section~\ref{sec:implementation}.
\end{remark}

\subsection{Modification of right-hand side}\label{sec:general_framework_e}
Finally, we must define $\tilde{F}$, a modification of $F$ in \eqref{eq:generalweak_time}, so that when the test function $v$ is chosen to be an auxiliary variable we recover the associated conservation or dissipation law.

More specifically, we require the construction of $\tilde{F} : \bbU \times \bbU^P \times \bbU \to \mathbb{R}$ with the following properties:
\begin{enumerate}
    \item $\tilde{F}(u, (\tilde{w}_p); v)$ is linear in its final argument.
    \item $\tilde{F}$ coincides with $F$ when evaluated at the associated test functions: for all $(u, v) \in \bbU \times \bbU$,
    \begin{equation}\label{eq:tildeFprop1}
        \tilde{F}(u, (w_p(u)); v) = F(u; v),
    \end{equation}
    when the left-hand side is well-defined.
    \item $\tilde{F}$ preserves the conservation/dissipation structures of $F$: for each $q = 1, \dots, P$,
      \begin{align}
      \label{eq:tildeFprop2}
      \text{if } F(u; w_q(u)) = 0, &\text{ then }\tilde{F}(u, (\tilde{w}_p); \tilde{w}_q) = 0; \nonumber \\
      \text{if } F(u; w_q(u)) \ge 0, &\text{ then }\tilde{F}(u, (\tilde{w}_p); \tilde{w}_q) \ge 0; \\
      \text{if } F(u; w_q(u)) \le 0, &\text{ then }\tilde{F}(u, (\tilde{w}_p); \tilde{w}_q) \le 0. \nonumber
      \end{align}
\end{enumerate}
This process is problem-specific, requires some judgement, and is best understood by example.

%\paragraph{Example (incompressible Navier--Stokes)}
For the incompressible Navier--Stokes equations, with the auxiliary variables defined, the choice $\bfv = \tilde{\bfu} \, (= \tilde{\bfw}_1)$ is now valid in \eqref{eq:generalweak_time}.
We wish to replicate the energy dissipation law \eqref{eq:incompnske} when this choice is made.
By inspection, if we define
\begin{equation}\label{eq:incompnsintermediateF}
    \tilde{F}(\bfu, \tilde{\bfu}; \bfv) \coloneqq (\tilde{\bfu} \times \curl\bfu, \bfv) - \frac{1}{\Re}(\nabla\tilde{\bfu}, \nabla\bfv),
\end{equation}
then when we test with $\bfv = \tilde{\bfu}$ in \eqref{eq:incompnsintermediateF}
\begin{equation}\label{eq:tildeFprop2_energy}
    \tilde{F}(\bfu, \tilde{\bfu}; \tilde{\bfu}) = - \frac{1}{\Re}\|\nabla\tilde{\bfu}\|^2 \le 0,
\end{equation}
satisfying \eqref{eq:tildeFprop2} for $q = 1$.
\added{Recalling that $\tilde{\bfomega}$ represents a discrete approximation to $\curl\bfu$,} to satisfy \eqref{eq:tildeFprop2} for $q = 2$ we further modify \eqref{eq:incompnsintermediateF} to recover the helicity law \eqref{eq:incompnshel} by defining
\begin{equation}\label{eq:incompnsfinalF}
    \tilde{F}(\bfu, (\tilde{\bfu}, \tilde{\bfomega}); \bfv) \coloneqq (\tilde{\bfu} \times \tilde{\bfomega}, \bfv) - \frac{1}{\Re}(\nabla\tilde{\bfu}, \nabla\bfv);
\end{equation}
thus, when further testing with $\bfv = \tilde{\bfomega} \, (= \tilde{\bfw}_2)$ in \eqref{eq:incompnsfinalF},
\begin{equation}\label{eq:tildeFprop2_helicity}
    \tilde{F}(\bfu, (\tilde{\bfu}, \tilde{\bfomega}); \tilde{\bfomega}) = - \frac{1}{\Re}(\nabla\tilde{\bfu}, \nabla\tilde{\bfomega}).
\end{equation}
In the ideal case $\Re = \infty$, both (\ref{eq:tildeFprop2_energy},~\ref{eq:tildeFprop2_helicity}) evaluate to zero, preserving the conservation structures.

\subsection{Construction of structure-preserving scheme}\label{sec:general_framework_f}
With $\tilde{F}$ defined, the final structure-preserving scheme is as follows.

\begin{definition}[Final discretization]
    Find $(u, (\tilde{w}_p)) \in \bbX_n \times \dot{\bbX}_n^P$ such that
    \begin{subequations}\label{eq:generalavcpgweak}
    \begin{align}
        \calI_n[M(u; \dot{u}, v)] &= \calI_n[\tilde{F}(u, (\tilde{w}_p); v)], \label{eq:generalavcpgweaka} \\
        \calI_n[M(u; v_p, \tilde{w}_p)] &= \int_{T_n}Q'_p(u; v_p), \label{eq:generalavcpgweakb}
    \end{align}
    \end{subequations}
    for all $(v, (v_p)) \in \dot{\bbX}_n \times \dot{\bbX}_n^P$.
\end{definition}

% \begin{remark}
%     \deleted{The exact integral $\int_{T_n}Q'_p(u; v_p)$ on the right-hand side of \eqref{eq:generalavcpgweakb} cannot, in general, be substituted for the approximate integral $\calI_n[Q'_p(u; v_p)]$ without breaking the structure-preserving properties of the scheme (outlined in Theorem~\ref{th:avcpg_sp} below). However, it is often not possible to evaluate this integral exactly. In such cases, we employ a quadrature rule of sufficiently high degree for this term that the quadrature error is on the order of machine precision.}
% \end{remark}

%\paragraph{Example (incompressible Navier--Stokes)}
For the incompressible Navier--Stokes equations, the final energy- and helicity-preserving scheme is as follows: find $(\bfu, (\tilde{\bfu}, \tilde{\bfomega})) \in \bbX_n \times (\dot{\bbX}_n)^2$ such that
\begin{subequations}\label{eq:incompnsavcpg}
\begin{align}
    \calI_n[(\dot{\bfu}, \bfv)]        &=  \calI_n\!\left[(\tilde{\bfu} \times \tilde{\bfomega}, \bfv) - \frac{1}{\Re} (\nabla\tilde{\bfu}, \nabla\bfv)\right]\!,  \label{eq:incompnsavcpg_a}  \\
    \calI_n[(\tilde{\bfu}, \bfv_1)]    &=  \int_{T_n}(\bfu, \bfv_1),  \label{eq:incompnsavcpg_b}  \\
    \calI_n[(\tilde{\bfomega}, \bfv_2)]  &=  \int_{T_n}(\curl\bfu, \bfv_2),  \label{eq:incompnsavcpg_c}
\end{align}
\end{subequations}
for all $(\bfv, (\bfv_1, \bfv_2)) \in \dot{\bbX}_n \times (\dot{\bbX}_n)^2$.

\subsection{Analysis: structure-preserving properties}

\begin{theorem}[Structure preservation of the framework]\label{th:avcpg_sp}
    Where solutions to \eqref{eq:generalavcpgweak} exist, they preserve the sign of the changes to the functionals $Q_q(u)$, $q=1, \dots, P$, across each timestep. In particular, if $Q_q(u)$ is conserved by the exact solution, then it is also conserved by the discretization, up to quadrature errors, solver tolerances, and machine precision.
\end{theorem}

\begin{proof}
    For each quantity of interest $Q_q$,
    \begin{equation}
            Q_q(u(t_{n+1})) - Q_q(u(t_n))
        =  \int_{T_n}Q'_q(u; \dot{u})
        =  \calI_n[M(u; \dot{u}, \tilde{w}_q)]
        =  \calI_n[\tilde{F}(u, (\tilde{w}_p); \tilde{w}_q)],
    \end{equation}
    where the second equality holds by \eqref{eq:generalavcpgweakb}, and the final equality by \eqref{eq:generalavcpgweaka}.
    Thus, if $Q_q(u)$ is conserved by the exact solution, $\tilde{F}(u, (\tilde{w}_p); \tilde{w}_q) = 0$ by \eqref{eq:tildeFprop2}; by the linearity of $\calI_n$, $Q_q(u)$ is conserved across timesteps.
    Otherwise, if $Q_q(u)$ is non-decreasing for the exact solution, $\tilde{F}(u, (\tilde{w}_p); \tilde{w}_q) \ge 0$ by \eqref{eq:tildeFprop2}; by the sign-preserving property of $\calI_n$ \eqref{eq:Isignpreserving}, $Q_q(u)$ is non-decreasing across timesteps.
    The same argument holds if $Q_q(u)$ is non-increasing.
\end{proof}

%\paragraph{Example (incompressible Navier--Stokes)}
For the scheme \eqref{eq:incompnsavcpg} for the incompressible Navier--Stokes equations, we find
\begin{subequations}
\begin{align}
    Q_1(\bfu(t_{n+1})) - Q_1(\bfu(t_n))  &=  - \frac{1}{\Re}\calI_n[\|\nabla\tilde{\bfu}\|^2]  \le  0,  \\
    Q_2(\bfu(t_{n+1})) - Q_2(\bfu(t_n))  &=  - \frac{1}{\Re}\calI_n[(\nabla\tilde{\bfu}, \nabla\tilde{\bfomega})].
\end{align}
\end{subequations}
These identities resemble weak forms of (\ref{eq:incompnske},~\ref{eq:incompnshel}).

\subsection{Numerical example: Hill spherical vortex}

To demonstrate these results, we consider a stationary Hill spherical vortex \cite{Hill_1894} with swirling motion \cite[Sec.~6(b)]{Moffatt_1969}.
In spherical coordinates $(r, \theta, \varphi)$, define the Stokes stream function
\begin{equation}
    \psi(r, \theta, \phi)  \coloneqq  \!\left\{\begin{aligned}
        &2\!\left(\frac{J_{\frac{3}{2}}(4\eta r)}{\!\left(4r\right)\!^{\frac{3}{2}}} - J_{\frac{3}{2}}(\eta)\!\right)\!(r\sin\theta)^2,  &  r &\le \frac{1}{4},  \\
        &0,  &  r &> \frac{1}{4},
    \end{aligned}\right.
\end{equation}
where $J_\alpha$ denotes the Bessel function of the first kind of order $\alpha$, and $\eta$ the first root of $J_{\frac{5}{2}}$, around $5.76$.
Up to projection onto $\bbU$, the initial conditions $\bfu(0)$ are given by $\psi$ as
\begin{equation}
    \bfu(0)  =  \frac{\partial_\theta \psi}{r^2\sin\theta}\hat{\bfr} - \frac{\partial_r \psi}{r\sin\theta}\hat{\bftheta} + \frac{4\eta\psi}{r\sin\theta}\hat{\bfvarphi}
\end{equation}
where $(\hat{\bfr}, \hat{\bftheta}, \hat{\bfvarphi})$ are the corresponding spherical unit vectors.
This defines the first stationary Hill spherical vortex of radius $\frac{1}{4}$; we consider the domain $\Omega = (-0.5, 0.5)^3$.

In space, we take $(\bbU, \bbP)$ to form the lowest order Taylor--Hood finite-element (FE) pair \cite[Sec.~54.3]{Ern_Guermond_2021b}, continuous Lagrange FEs \cite[Sec.~6~\&~7]{Ern_Guermond_2021a} of orders 2 and 1 respectively; we use tetrahedral cells of uniform diameter $2^{-3}$.
In time, we take $S = 3$, with $\calI_n$ the exact integral, a uniform timestep $\Delta t = 2^{-10}$ and duration $3\cdot2^{-6}$.
For comparison, we run simulations using the full structure-preserving scheme \eqref{eq:incompnsavcpg} alongside one that preserves the structure in the energy only using $\tilde{F}$ as defined in \eqref{eq:incompnsintermediateF}.
We vary $\Re \in 2^{2s}$ over $s \in \{0, \dots, 8\}$.

Fig.~\ref{fig:hill_invariants} shows the evolution of the energy $Q_1$ and helicity $Q_2$ in the two simulations.
From the lower-right graph, we observe that the solely energy-preserving scheme has an artificial dissipation in the helicity at all $\Re$, due to the lack of preservation of the update law for helicity.
In all other cases, the dissipations in the energy and helicity decrease in magnitude as $\Re$ increases; moreover the energy is universally non-increasing.
Fig.~\ref{fig:hill_plots} shows a cross-section of the velocity streamlines at the initial and final times with both schemes, at $\Re = 2^{16}$.
When compared with the results from the full structure-preserving scheme, one can observe that the artificial helicity dissipation in the energy-preserving scheme causes increased unphysical instability in the vortex.

\begin{figure}[!ht]
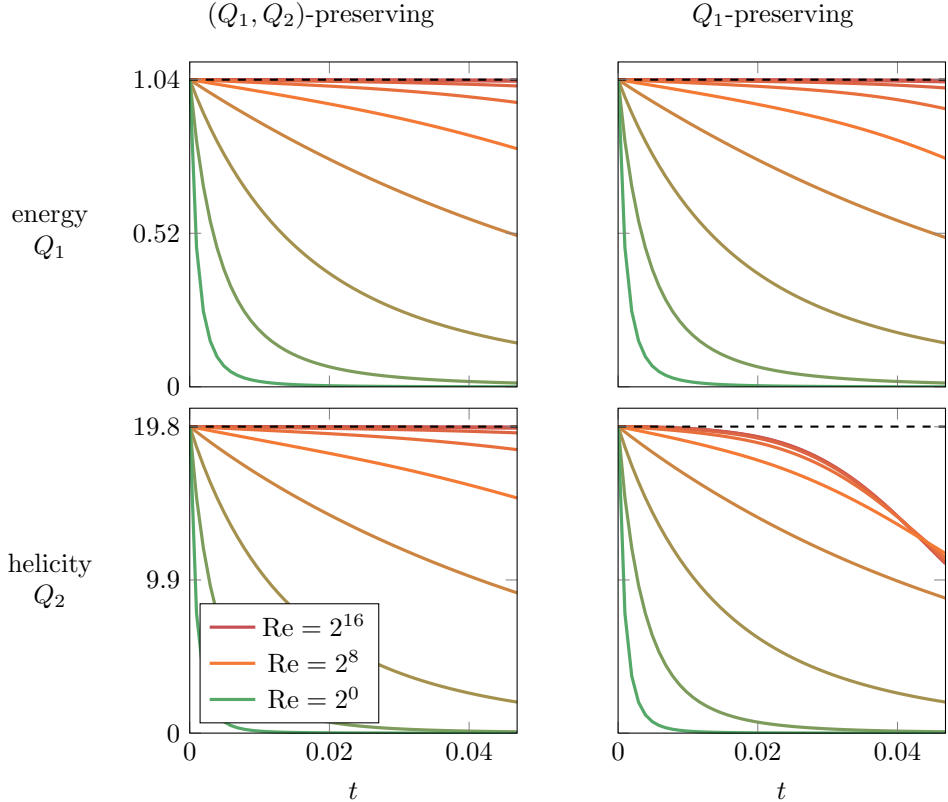

    \captionsetup[subfigure]{justification = centering}
    \centering

    \begin{subfigure}{0.09\textwidth}
        \centering

        $\,$
    \end{subfigure}%
    \begin{subfigure}{0.455\textwidth}
        \centering

        $(Q_1, Q_2)$-preserving
    \end{subfigure}%
    \begin{subfigure}{0.455\textwidth}
        \centering

        $Q_1$-preserving
    \end{subfigure}

    \vspace{4mm}

    \begin{subfigure}{0.09\textwidth}
        \centering

        energy \\ $Q_1$

        \raisebox{19mm}{}
    \end{subfigure}%
    \begin{subfigure}{0.455\textwidth}
        \centering
        
        \begin{tikzpicture}
        \begin{axis}[
            xmin = 0,  xmax = 0.046875,  xticklabel = \empty,  scaled x ticks = false,
            ymin = 0,  ymax = 1.1,       ytick = {0, 0.52, 1.04},  yticklabels = {$0$, $0.52$, $1.04$},
            width = 0.99\textwidth,  height = 0.99\textwidth,
            axis on top,
            legend pos = south west,
        ]
            \input{plots/hill_invariants/energy/andrews_farrell/16.tex}
            \input{plots/hill_invariants/energy/andrews_farrell/14.tex}
            \input{plots/hill_invariants/energy/andrews_farrell/12.tex}
            \input{plots/hill_invariants/energy/andrews_farrell/10.tex}
            \input{plots/hill_invariants/energy/andrews_farrell/8.tex}
            \input{plots/hill_invariants/energy/andrews_farrell/6.tex}
            \input{plots/hill_invariants/energy/andrews_farrell/4.tex}
            \input{plots/hill_invariants/energy/andrews_farrell/2.tex}
            \input{plots/hill_invariants/energy/andrews_farrell/0.tex}
            \input{plots/hill_invariants/energy/constant.tex}
            % \legend{$\Re = 2^{16}$, $\Re = 2^8$, $\Re = 2^0$}
        \end{axis}
        \end{tikzpicture}
    \end{subfigure}%
    \begin{subfigure}{0.455\textwidth}
        \centering
        
        \begin{tikzpicture}
        \begin{axis}[
            xmin = 0,  xmax = 0.046875,  xticklabel = \empty,  scaled x ticks = false,
            ymin = 0,  ymax = 1.1,       ytick = {0, 0.52, 1.04},  yticklabel = \empty,
            width = 0.99\textwidth,  height = 0.99\textwidth,
            axis on top,
            legend pos = south west,
        ]
            \input{plots/hill_invariants/energy/cpg/16.tex}
            \input{plots/hill_invariants/energy/cpg/14.tex}
            \input{plots/hill_invariants/energy/cpg/12.tex}
            \input{plots/hill_invariants/energy/cpg/10.tex}
            \input{plots/hill_invariants/energy/cpg/8.tex}
            \input{plots/hill_invariants/energy/cpg/6.tex}
            \input{plots/hill_invariants/energy/cpg/4.tex}
            \input{plots/hill_invariants/energy/cpg/2.tex}
            \input{plots/hill_invariants/energy/cpg/0.tex}
            \input{plots/hill_invariants/energy/constant.tex}
            % \legend{$\Re = 2^{16}$, $\Re = 2^8$, $\Re = 2^0$}
        \end{axis}
        \end{tikzpicture}
    \end{subfigure}

    \begin{subfigure}{0.09\textwidth}
        \centering

        helicity \\ $Q_2$

        \raisebox{26mm}{}
    \end{subfigure}%
    \begin{subfigure}{0.455\textwidth}
        \centering
        
        \begin{tikzpicture}
        \begin{axis}[
            xmin = 0,  xmax = 0.046875,  xtick = {0, 0.02, 0.04},  xticklabels = {$0$, $0.02$, $0.04$},  scaled x ticks = false,  xlabel = {$t$},
            ymin = 0,  ymax = 21,        ytick = {0, 9.9, 19.8},   yticklabels = {$0$, $9.9$, $19.8$},
            width = 0.99\textwidth,  height = 0.99\textwidth,
            axis on top,
            legend pos = south west,
        ]
            \input{plots/hill_invariants/helicity/andrews_farrell/16.tex}
            \input{plots/hill_invariants/helicity/andrews_farrell/14.tex}
            \input{plots/hill_invariants/helicity/andrews_farrell/12.tex}
            \input{plots/hill_invariants/helicity/andrews_farrell/10.tex}
            \input{plots/hill_invariants/helicity/andrews_farrell/8.tex}
            \input{plots/hill_invariants/helicity/andrews_farrell/6.tex}
            \input{plots/hill_invariants/helicity/andrews_farrell/4.tex}
            \input{plots/hill_invariants/helicity/andrews_farrell/2.tex}
            \input{plots/hill_invariants/helicity/andrews_farrell/0.tex}
            \input{plots/hill_invariants/helicity/constant.tex}
            \legend{$\Re = 2^{16}$, $\Re = 2^8$, $\Re = 2^0$}
        \end{axis}
        \end{tikzpicture}
    \end{subfigure}%
    \begin{subfigure}{0.455\textwidth}
        \centering
        
        \begin{tikzpicture}
        \begin{axis}[
            xmin = 0,  xmax = 0.046875,  xtick = {0, 0.02, 0.04},  xticklabels = {$0$, $0.02$, $0.04$},  scaled x ticks = false,  xlabel = {$t$},
            ymin = 0,  ymax = 21,        ytick = {0, 9.9, 19.8},  yticklabel = \empty,
            width = 0.99\textwidth,  height = 0.99\textwidth,
            axis on top,
            legend pos = south west,
        ]
            \input{plots/hill_invariants/helicity/cpg/16.tex}
            \input{plots/hill_invariants/helicity/cpg/14.tex}
            \input{plots/hill_invariants/helicity/cpg/12.tex}
            \input{plots/hill_invariants/helicity/cpg/10.tex}
            \input{plots/hill_invariants/helicity/cpg/8.tex}
            \input{plots/hill_invariants/helicity/cpg/6.tex}
            \input{plots/hill_invariants/helicity/cpg/4.tex}
            \input{plots/hill_invariants/helicity/cpg/2.tex}
            \input{plots/hill_invariants/helicity/cpg/0.tex}
            \input{plots/hill_invariants/helicity/constant.tex}
            % \legend{$\Re = 2^{16}$, $\Re = 2^8$, $\Re = 2^0$}
        \end{axis}
        \end{tikzpicture}
    \end{subfigure}

    \caption{Evolution of the energy $Q_1$ and helicity $Q_2$ in the $(Q_1, Q_2)$-preserving scheme \eqref{eq:incompnsavcpg} and the $Q_1$-preserving scheme derived from \eqref{eq:incompnsintermediateF}, with varying $\Re = 2^{2s}$ for $s \in \{0, \dots, 8\}$.}

    \label{fig:hill_invariants}
\end{figure}

\begin{figure}
    \captionsetup[subfigure]{justification = centering}
    \centering

    \begin{subfigure}{0.5\textwidth}
        \centering

        \includegraphics[width=\textwidth]{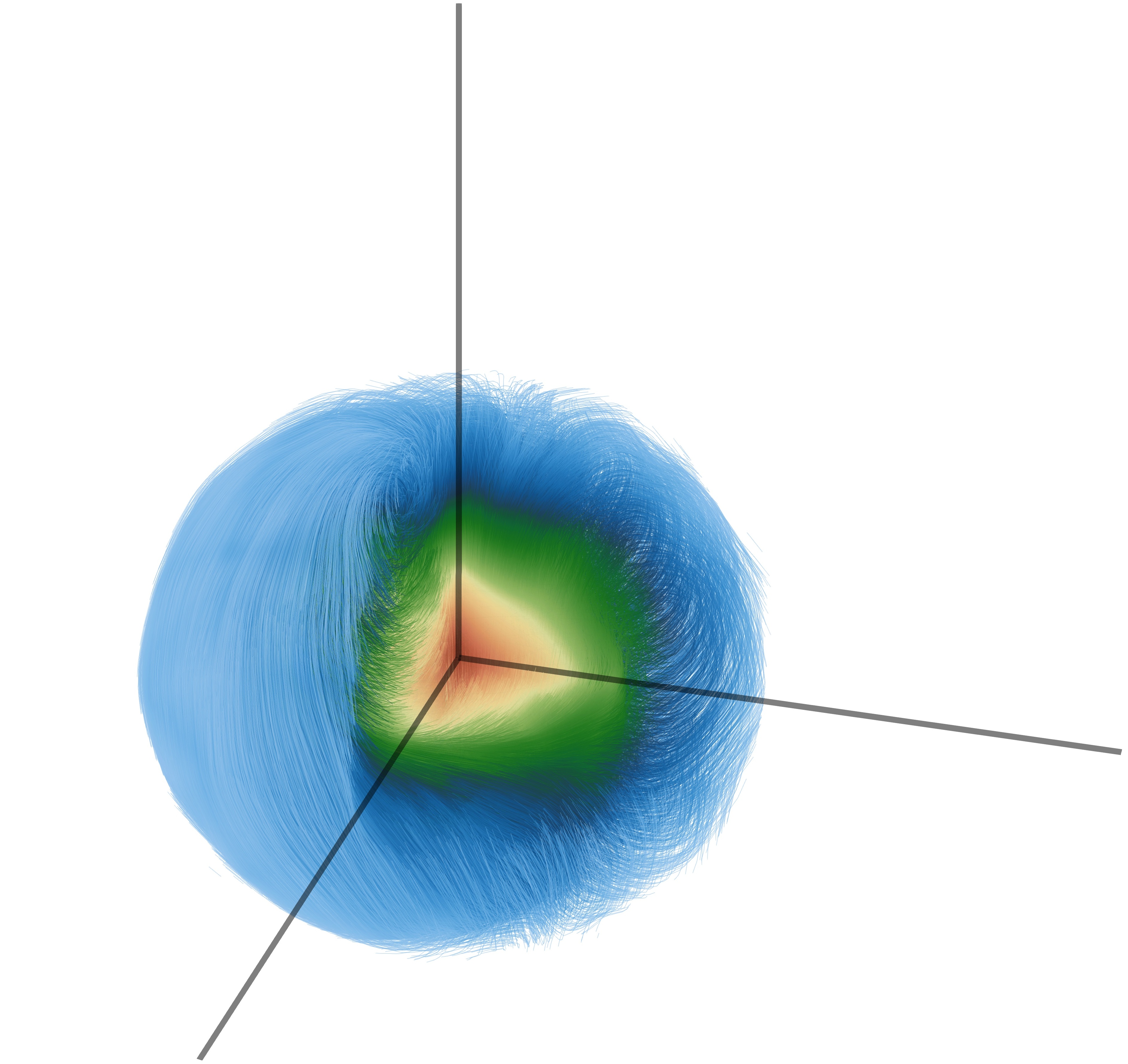}

        \caption{$t = 0$}
    \end{subfigure}

    \vspace{0mm}

    \begin{subfigure}{0.5\textwidth}
        \centering

        \includegraphics[width=\textwidth]{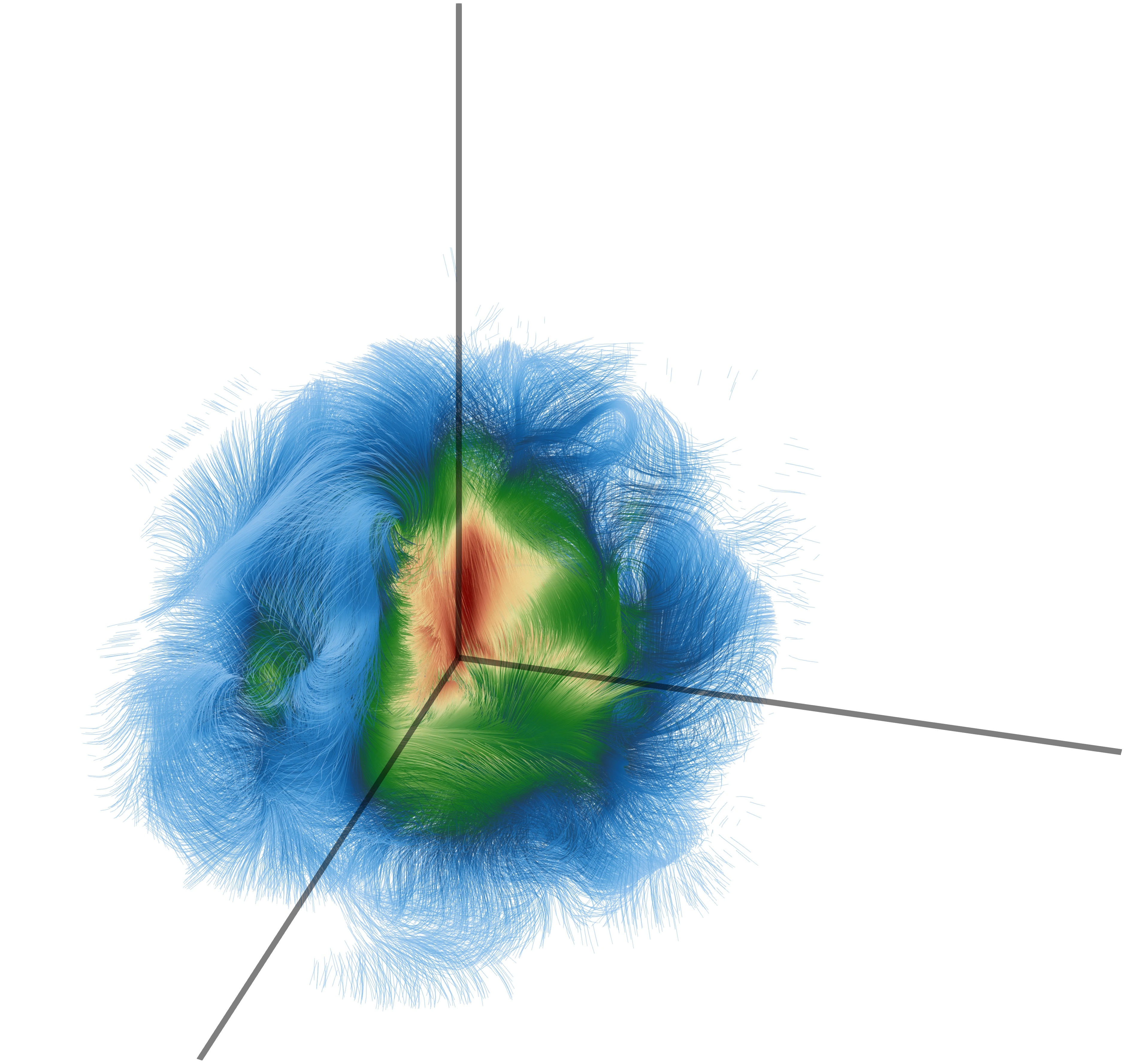}

        \caption{$(Q_1, Q_2)$-preserving, $t = 3 \cdot 2^{-6}$}
    \end{subfigure}%
    \begin{subfigure}{0.5\textwidth}
        \centering

        \includegraphics[width=\textwidth]{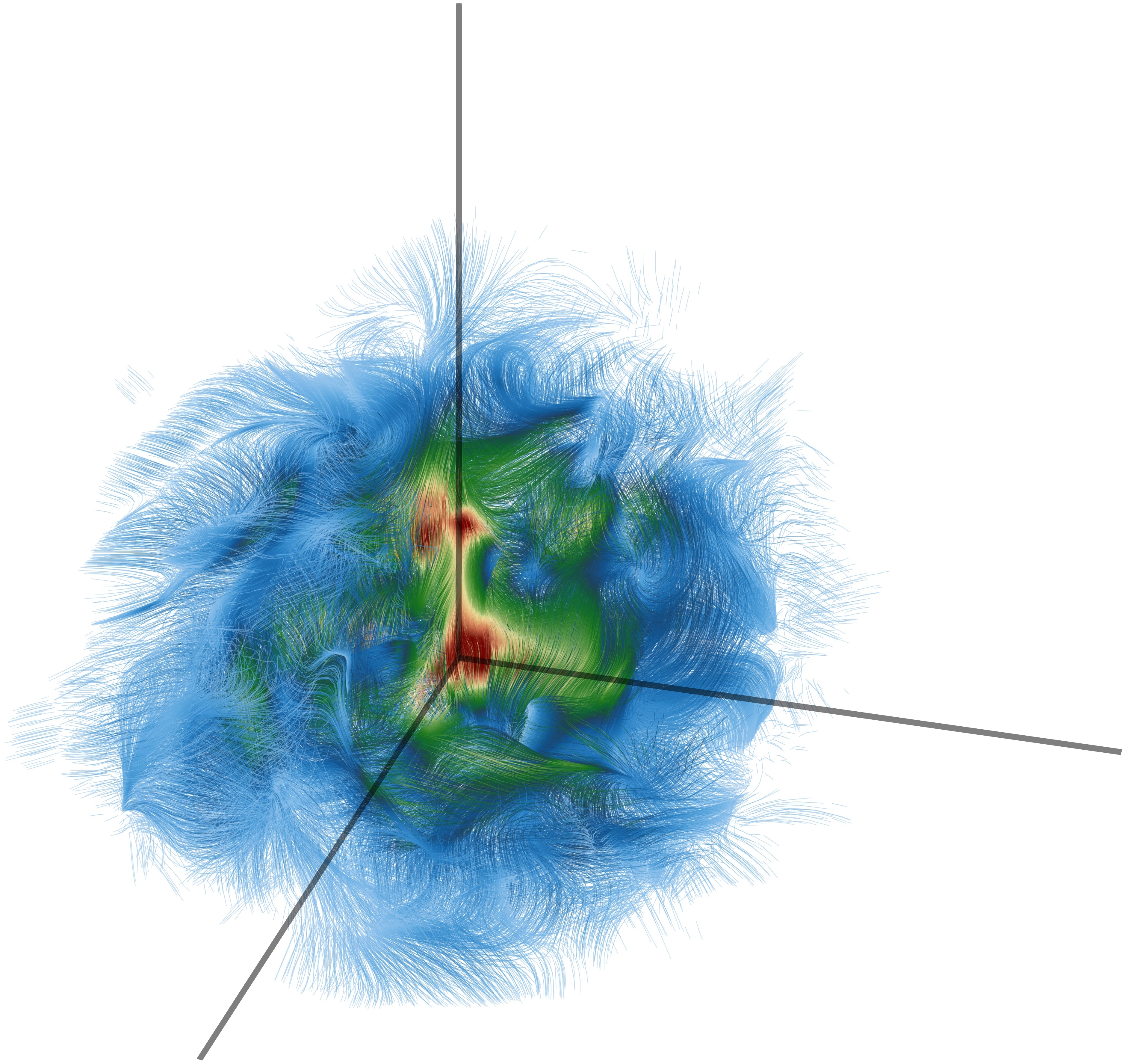}

        \caption{$Q_1$-preserving, $t = 3 \cdot 2^{-6}$}
    \end{subfigure}

    \caption{Cross-sections of streamlines of the velocity $\bfu$ for the Hill vortex at times $t \in \{0, 3 \cdot 2^{-6}\}$ in the $(Q_1, Q_2)$-preserving scheme \eqref{eq:incompnsavcpg} and the $Q_1$-preserving scheme derived from \eqref{eq:incompnsintermediateF} with $\Re = 2^{16}$. Coloring indicates $\|\bfu\|$.}

    \label{fig:hill_plots}
\end{figure}

\section{Compressible Navier--Stokes} \label{sec:compressible}
We now consider structure-preserving schemes for the compressible Navier--Stokes equations.
We seek a scheme that will conserve the mass, momentum and energy, and preserve the behaviour of the entropy;
specifically, we would like the entropy to be conserved in the ideal limit, and non-decreasing otherwise, i.e.~the scheme should obey the second law of thermodynamics.\footnote{
    \added{Note, the conservation of entropy in the ideal limit may not always be appropriate.
    This is in particular true of less regular solutions, such as vanishing-viscosity limits with shocks.}
}
The non-dissipation of entropy is a crucial aspect in the analysis and behaviour of solutions to the compressible Navier--Stokes equations \cite{Feireisl_et_al_2021} with quantitative implications on the regularity of solutions and qualitative implications on the dissipation rate; it is therefore highly desirable to preserve it.

Over a bounded Lipschitz polyhedral domain $\Omega \subset \bbR^d$, $d \in \{1, 2, 3\}$, the compressible Navier--Stokes equations can be written in the following non-dimensionalized form:
\begin{subequations}\label{eq:compns1}
\begin{align}
           \dot{\rho}  &=  - \div[\rho\bfu],  \\
       \rho\dot{\bfu}  &=  - \rho\bfu\cdot\nabla\bfu - \nabla p + \div\!\left[\frac{2}{\Re}\rho\tau[\bfu]\right]\!,  \\
    \dot{\varepsilon}  &=  - \div[\varepsilon\bfu] - p\div\bfu + \frac{2}{\Re}\upsilon[\bfu, \bfu] + \div\!\left[\frac{1}{\Re\Pr}\rho\nabla\theta\right]\!.
\end{align}
\end{subequations}
Here, $\rho$, $p$, $\bfu$, $\varepsilon$ and $\theta$ are the density, pressure, velocity, internal energy density and temperature respectively, $\Re > 0$ and $\Pr > 0$ are the Reynolds and Prandtl numbers, potentially functions of $\rho$ and $\varepsilon$, the deviatoric strain $\tau : \bbR^d \to \bbR^d_{\rm sym}$ is defined
\begin{subequations}
\begin{equation}
    \tau[\bfu]  \coloneqq  \frac{1}{2}\nabla\bfu + \frac{1}{2}\nabla\bfu^\top - \frac{1}{3}(\div\bfu)I,
\end{equation}
trace-free when $d=3$, and the positive semi-definite bilinear form $\upsilon : \bbR^d \times \bbR^d \to \bbR$ is defined
\begin{equation}
    \upsilon[\bfu, \bfv]  \coloneqq  \!\left(\frac{1}{2}\nabla\bfu + \frac{1}{2}\nabla\bfu^\top\right)\!:\!\left(\frac{1}{2}\nabla\bfv + \frac{1}{2}\nabla\bfv^\top\right)\! - \frac{1}{3}(\div\bfu)(\div\bfv).
\end{equation}
\end{subequations}
We assume the Stokes hypothesis, that the bulk viscosity is zero \cite{Stokes_1845}. This is for brevity only and is not necessary; the ideas we present in this section readily extend to more complex stress tensors.
For simplicity, we assume periodic BCs.

The system \eqref{eq:compns1} is completed by constitutive relations relating two of $\rho$, $\theta$, $p$, $\varepsilon$ to the others.
Our discretization is independent of these, but for concreteness we will employ as running example an ideal gas, with constitutive relations
\begin{equation}\label{eq:ideal_state}
              p  =  \rho\theta,  \quad
    \varepsilon  =  C_Vp,
\end{equation}
where $C_V$ is the non-dimensionalized heat capacity at constant volume, ${3}/{2}$ for a monatomic gas.

We also define the inverse temperature $\beta \coloneqq {\theta^{-1}}$,
and the specific entropy $s$, corresponding to the total entropy $\int_\Omega \rho s$, satisfying the (intensive) fundamental thermodynamic relation
\begin{subequations}
\begin{equation}\label{eq:ftr1}
    \beta \rmd\varepsilon  =  \rmd[\rho s] - g \rmd\rho.
\end{equation}
Here, $g = s - (\varepsilon + p)\beta/\rho$ is the negation of the specific free energy, or Gibbs free energy per unit mass, per unit temperature.
Taking differentials gives the second thermodynamic relation,
\begin{equation}\label{eq:ftr2}
    \rho \rmd g + \varepsilon \rmd\beta + \rmd[p\beta]  =  0.
\end{equation}
\end{subequations}
For an ideal gas, $s$ and $g$ evaluate as
\begin{equation}\label{eq:ideal_gibbs}
    s  =  \log\!\left(\frac{\theta^{C_V}}{\rho}\right)\!,  \quad
    g  =  s - (C_V + 1).
\end{equation}

\subsection{Definition of semi-discrete form}
To define the semi-discrete form, we must first choose a convenient parametrization.
Many options are available here, such as primitive or conservative variables.
We shall choose $\sigma = \rho^{\frac{1}{2}}$, $\bfmu = \rho^{\frac{1}{2}}\bfu$, and $\zeta = \log(\varepsilon)$.
This parametrization is chosen with some hindsight.
The choice of $\bfmu$ ensures the energy is independent of the density, limiting the number of auxiliary variables that must be introduced; the choice of $\sigma$ balances this in a way that later simplifies the conservation of momentum; the choice of $\zeta$ ensures the internal energy remains positive.
Writing \eqref{eq:compns1} in terms of $\sigma$, $\bfmu$, $\zeta$ yields
\begin{subequations}\label{eq:compns2}
\begin{align}
    (\dot{\rho} =)         \quad     2\sigma\dot{\sigma}  &=  - \div[\rho\bfu],  \\
                                       \sigma\dot{\bfmu}  &=  - \frac{1}{2}(\rho\bfu\cdot\nabla\bfu + \div[\rho\bfu^{\otimes 2}]) - \nabla p + \div\!\left[\frac{2}{\Re}\rho\tau[\bfu]\right]\!,  \label{eq:compns2b}  \\
    (\dot{\varepsilon} =)  \quad  \varepsilon\dot{\zeta}  &=  - \div[\varepsilon\bfu] - p\div\bfu + \frac{2}{\Re}\rho\upsilon[\bfu, \bfu] + \div\!\left[\frac{1}{\Re\Pr}\rho\nabla\theta\right]\!,
\end{align}
\end{subequations}
where $\rho = \sigma^2$, $\bfu = \sigma^{-1}\bfmu$, $\varepsilon = \exp(\zeta)$, and it is assumed that known constitutive relations determine $p$, $\theta$ as functions of $\rho$, $\varepsilon$.
For some continuous, spatially periodic FE space $\bbV \subset C^0_{\text{per}}(\Omega)$, we define the mixed FE space $\bbU \coloneqq \bbV^{1+d+1}$; we use the same space for each variable both for simplicity, and as it will help in ensuring momentum conservation\footnote{Discontinuous spaces $\bbV \not\subset C^0(\Omega)$ are often preferred for discretization. This necessitates the introduction of facet and penalty terms to handle the non-conformity; such an extension is possible, but omitted here for brevity.}.
We may then define a semi-discrete variational problem: find $(\sigma, \bfmu, \zeta) \in \bbX$, for $\bbX$ defined as in \eqref{eq:affinespace}, such that
\begin{equation}\label{eq:compns_semidisc}
      M((\sigma, \zeta); (\dot{\sigma}, \dot{\bfmu}, \dot{\zeta}), (v_\rho, \bfv_m, v_\varepsilon))
    = F((\sigma, \bfmu, \zeta); (v_\rho, \bfv_m, v_\varepsilon))
\end{equation}
at all times $t \in \bbR_+$ and for all $(v_\rho, \bfv_m, v_\varepsilon) \in \bbU$, where $M$, $F$ are defined
\begin{subequations}
\begin{align}
                M
    &\coloneqq  \int_\Omega 2\sigma\dot{\sigma}v_\rho + \int_\Omega \sigma\dot{\bfmu}\cdot\bfv_m + \int_\Omega \varepsilon\dot{\zeta}v_\varepsilon,  \\
                F
    &\coloneqq  \int_\Omega \rho\bfu\cdot\nabla v_\rho  \\
    &\qquad     + \int_\Omega \!\left[\frac{1}{2}\rho\bfu\cdot(\nabla \bfv_m\cdot\bfu - \nabla \bfu\cdot\bfv_m) - \bfv_m\cdot\nabla p - \frac{2}{\Re}\rho\upsilon[\bfu, \bfv_m]\right]\!  \nonumber   \\
    &\qquad     + \int_\Omega \!\left[\bfu\cdot(\varepsilon\nabla v_\varepsilon + \nabla[pv_\varepsilon]) + \frac{2}{\Re}\rho\upsilon[\bfu, \bfu]v_\varepsilon - \frac{1}{\Re\Pr}\rho\nabla\theta\cdot\nabla v_\varepsilon\right]\!. \nonumber
\end{align}
\end{subequations}

In particular, for an ideal gas, writing the equations of state \eqref{eq:ideal_state} in terms of 
$\rho = \sigma^2$ and $\varepsilon = \exp(\zeta)$ yields
\begin{equation}
    p       =  \frac{\varepsilon}{C_V} = \frac{\exp{(\zeta)}}{C_V},  \quad
    \theta  =  \frac{p}{\rho} = \frac{p}{\sigma^2}.
\end{equation}

\subsection{Definition of timestepping scheme}
Over the timestep $T_n$, choosing $\calI_n$ to be the exact integral casts \eqref{eq:compns_semidisc} into a fully discrete form (i.e.~a \added{continuous Petrov--Galerkin} \deleted{CPG} discretization): find $(\sigma, \bfmu, \zeta) \in \bbX_n$ such that
\begin{equation}\label{eq:compns_disc}
      \int_{T_n}M((\sigma, \zeta); (\dot{\sigma}, \dot{\bfmu}, \dot{\zeta}), (v_\rho, \bfv_m, v_\varepsilon))
    = \int_{T_n}F((\sigma, \bfmu, \varepsilon); (v_\rho, \bfv_m, v_\varepsilon)),
\end{equation}
for all $(v_\rho, \bfv_m, v_\varepsilon) \in \dot{\bbX}_n$, with $\bbX_n$ defined as in \eqref{eq:solution_space}.
For simplicity, we assume that a sufficiently small timestep and fine mesh are chosen so that $\sigma$ remains positive, implying the constitutive relations remain well-defined.

\subsection{Identification of associated test functions}
Including each component of the momentum, we have $3+d$ quantities of interest,
\begin{equation}
    Q_1     \coloneqq  \int_\Omega \sigma^2,  \quad
    \bfQ_2  \coloneqq  \int_\Omega \sigma\bfmu,  \quad
    Q_3     \coloneqq  \int_\Omega \!\left[\frac{1}{2}\|\bfmu\|^2 + \varepsilon\right]\!,  \quad
    Q_4     \coloneqq  \int_\Omega \rho s,
\end{equation}
the mass, momentum, energy, and entropy respectively, where $s$ is a function of $\rho = \sigma^2$ and $\varepsilon = \exp(\zeta)$.
By evaluating the Fr\'echet derivatives, we identify these with the respective associated test functions
\begin{equation}\label{eq:compns_atf}
    (1, \bfzero, 0),  \quad
    (\frac{1}{2}\bfu, I, 0),  \quad
    (0, \bfu, 1),  \quad
    (g, \bfzero, \beta),
\end{equation}
where again $\beta = {\theta^{-1}}$.
Since $\bfQ_2$ contains $d$ quantities of interest, one can more precisely state that the associated test functions for each component $Q_{2i}$ of the momentum are $(\frac{1}{2}u_i, \bfe_i, 0)$. 
The associated test functions for $Q_4$ are found from \eqref{eq:ftr1}.

\subsection{Introduction of auxiliary variables}
We introduce auxiliary variables for each of the associated test functions in \eqref{eq:compns_atf}, according to \eqref{eq:av_definition}.
It is straightforward to show that, provided $1 \in \bbV$, the auxiliary variables that would be introduced to approximate $0$, $1$, $I$ are precisely $0$, $1$, $I$; it is therefore unnecessary to introduce them.

The remaining associated test functions include two for $\bfu$, and one each for $g$, $\beta$.
The variational relations \eqref{eq:av_definition} satisfied by each of the auxiliary variables for $\bfu$ are identical, and so these two auxiliary variables are identical.
This leaves three auxiliary variables, $(\tilde{g}, \tilde{\bfu}, \tilde{\beta}) \in \dot{\bbX}_n$, satisfying
\begin{equation}\label{eq:compns_avs}
      \int_{T_n}M((\sigma, \zeta); (v_g, \bfv_u, v_\beta), (\tilde{g}, \tilde{\bfu}, \tilde{\beta}))
    = \int_{T_n}M((\sigma, \zeta); (v_g, \bfv_u, v_\beta), (g, \bfu, \beta)),
\end{equation}
for all $(v_g, \bfv_u, v_\beta) \in \dot{\bbX}_n$, where again $\beta = {\theta}^{-1}$, $g$ are functions of $\rho = \sigma^2$, $\varepsilon = \exp(\zeta)$, and $\bfu = {\sigma}^{-1}\bfmu$.
Like $\beta$, we assume that $\tilde{\beta} > 0$.

For an ideal gas, from \eqref{eq:ideal_gibbs} the negative specific free energy per unit temperature $g$ can be defined in terms of $\rho = \sigma^2$, $\varepsilon = \exp(\zeta)$ as
\begin{equation}
    g  =  \log\!\left(\frac{\varepsilon^{C_V}}{\rho^{C_V+1}}\right)\! - (C_V + 1 + C_V\log C_V).
\end{equation}

\subsection{Modification of right-hand side}
We now introduce $\tilde{g}$, $\tilde{\beta}$ into $F$.
Just as we consider $\rho$, $p$, and $\varepsilon$ to be functions of $\sigma$ and $\zeta$,
let $\tilde{\rho}$, $\tilde{p}$, and $\tilde{\varepsilon}$ denote an auxiliary density, pressure and energy density, determined by the fluid's constitutive relations as functions of the auxiliary inverse temperature $\tilde{\beta}$ and auxiliary negative specific free energy per unit temperature $\tilde{g}$.
Crucially, in this sense $\tilde{\rho}(\tilde{g}, \tilde{\beta})$ differs from $\rho = \sigma^2$, and $\tilde{\varepsilon}(\tilde{g}, \tilde{\beta})$ from $\varepsilon = \exp(\zeta)$.
By inspection, we define $\tilde{F}((\sigma, \bfmu, \zeta), (\tilde{g}, \tilde{\bfu}, \tilde{\beta}); (v_\rho, \bfv_m, v_u))$ to be
\begin{align}
                \tilde{F}
    &\coloneqq  \int_\Omega \tilde{\rho}\tilde{\bfu}\cdot\nabla v_\rho  \\
    &\qquad     + \int_\Omega \!\left[\frac{1}{2}\tilde{\rho}\tilde{\bfu}\cdot(\nabla \bfv_m \cdot \tilde{\bfu} - \nabla \tilde{\bfu} \cdot \bfv_m) - \bfv_m\cdot\nabla\tilde{p} - \frac{2}{\Re}\rho\upsilon[\tilde{\bfu}, \bfv_m]\right]\!  \notag   \\
    &\qquad     + \int_\Omega \!\left[\tilde{\bfu}\cdot(\tilde{\varepsilon}\nabla v_\varepsilon + \nabla[\tilde{p}v_\varepsilon]) + \frac{2}{\Re}\rho\upsilon[\tilde{\bfu}, \tilde{\bfu}]v_\varepsilon + \frac{1}{\Re\Pr}\rho\theta^2\nabla\tilde{\beta}\cdot\nabla v_\varepsilon\right]\!.  \notag
\end{align}
Substituting $(v_\rho, \bfv_m, v_\varepsilon)$ for each set of auxiliary variables for each quantity of interest,
\begin{subequations}
\begin{align}
    \tilde{F}(\dots; (1, \bfzero, 0))                      &=  0,  \\
    \tilde{F}(\dots; (\frac{1}{2}\tilde{\bfu}, I, 0))      &=  0,  \\
    \tilde{F}(\dots; (0, \tilde{\bfu}, 1))                 &=  0,  \\
    \tilde{F}(\dots; (\tilde{g}, \bfzero, \tilde{\beta}))  &=  \frac{1}{\Re}\int_{\Omega}\rho\!\left(\tilde{\beta}\upsilon[\tilde{\bfu}, \tilde{\bfu}] + \frac{1}{\Pr}\theta^2\|\nabla\tilde{\beta}\|^2\right)\!  \ge  0.  \label{eq:ftilde_q4}
\end{align}
\end{subequations}
These identities are immediate by evaluation of the left-hand side.
The evaluation of \eqref{eq:ftilde_q4} includes the integral
\begin{equation}\label{eq:awkward_integral}
    \int_\Omega \tilde{\bfu}\cdot\!\left(\tilde{\rho}\nabla\tilde{g} + \tilde{\varepsilon}\nabla\tilde{\beta} + \nabla[\tilde{p}\tilde{\beta}]\right)\!.
\end{equation}
Any set of intensive thermodynamic quantities satisfying a valid constitutive law must satisfy the thermodynamic relation \eqref{eq:ftr2}.
As $(\tilde{\rho}, \tilde{g}, \tilde{\varepsilon}, \tilde{\beta}, \tilde{p})$ are constructed to satisfy such a law, we see $\tilde{\rho}\nabla\tilde{g} + \tilde{\varepsilon}\nabla\tilde{\beta} + \nabla[\tilde{p}\tilde{\beta}] = \bfzero$ everywhere, and \eqref{eq:awkward_integral} must evaluate to zero.

In particular, for an ideal gas, the auxiliary $\tilde{\rho}$, $\tilde{p}$, $\tilde{\varepsilon}$ can be written explicitly in terms of $\tilde{g}$, $\tilde{\beta}$ as
\begin{equation}
    \tilde{\rho}         =  \tilde{\beta}^{- C_V}\exp\!\left(- \tilde{g} - (C_V + 1)\right)\!,  \quad
    \tilde{p}            =  \frac{\tilde{\rho}}{\tilde{\beta}},  \quad
    \tilde{\varepsilon}  =  C_V\tilde{p}.
\end{equation}

\subsection{Construction of structure-preserving scheme}
    The final structure-preserving scheme is as follows: find $((\sigma, \bfmu, \zeta), (\tilde{g}, \tilde{\bfu}, \tilde{\beta})) \in \bbX_n \times \dot{\bbX}_n$ such that
    \begin{subequations}\label{eq:compns_avcpg}
    \begin{align}
        \int_{T_n}M((\sigma, \zeta); (\dot{\sigma}, \dot{\bfmu}, \dot{\zeta}), (v_\rho, \bfv_m, v_\varepsilon))
            &=  \int_{T_n}\tilde{F}((\sigma, \bfmu, \zeta), (\tilde{g}, \tilde{\bfu}, \tilde{\beta}); (v_\rho, \bfv_m, v_\varepsilon)),  \\
        \int_{T_n}M((\sigma, \zeta); (v_g, \bfv_u, v_\beta), (\tilde{g}, \tilde{\bfu}, \tilde{\beta}))
            &=  \int_{T_n}M((\sigma, \zeta); (v_g, \bfv_u, v_\beta), (g, \bfu, \beta)),
    \end{align}
    \end{subequations}
    for all $((v_\rho, \bfv_m, v_\varepsilon), (v_g, \bfv_u, v_\beta)) \in \dot{\bbX}_n \times \dot{\bbX}_n$, where $g$, $\beta$ are functions of $\rho = \sigma^2$, $\varepsilon = \exp(\zeta)$, and $\bfu = \frac{1}{\sigma}\bfmu$.
    Assuming a solution to \eqref{eq:compns_avcpg} exists, it necessarily exhibits all desired structure-preserving properties.

\subsection{Structure-preserving properties}

For the discrete scheme \eqref{eq:compns_avcpg} for the compressible Navier--Stokes equations, we find
\begin{subequations}
\begin{align}
    Q_1(\bfu(t_{n+1})) - Q_1(\bfu(t_n))  &=  0,  \\
    \bfQ_2(\bfu(t_{n+1})) - \bfQ_2(\bfu(t_n))  &=  \bfzero,  \\
    Q_3(\bfu(t_{n+1})) - Q_3(\bfu(t_n))  &=  0,  \\
    Q_4(\bfu(t_{n+1})) - Q_4(\bfu(t_n))  &=  - \frac{1}{\Re}\int_{T_n}\int_{\Omega}\rho\!\left(\tilde{\beta}\upsilon[\tilde{\bfu}, \tilde{\bfu}] + \frac{1}{\Pr}\theta^2\|\nabla\tilde{\beta}\|^2\right)\!  \ge  0.
\end{align}
\end{subequations}

\subsection{Numerical examples: Euler and supersonic test}

We consider two numerical examples: an inviscid Euler test to demonstrate entropy conservation, and a viscous test with a supersonic initial condition in velocity.
In both cases we take $C_V = 5/2$, typical for air at room temperature, and set $\Omega$ to be the unit square $(0,1)^2$.
For $\bbV$ we use piecewise linear continuous Lagrange FEs \cite[Sec.~6~\&~7]{Ern_Guermond_2021a}.
We compare the scheme \eqref{eq:compns_avcpg} at $S = 1$ (i.e.~at lowest order in time) with an implicit midpoint discretization of \eqref{eq:compns2}.

\subsubsection{Inviscid (Euler) test} \label{sec:inviscid_test}

We first consider an adiabatic (uniform $s$) perturbation in the state functions $\sigma$, $\zeta$, with $\Re = \infty$, i.e.~discarding viscous and thermally dissipative terms.
We take the initial conditions to be
\begin{subequations}
\begin{align}
    \sigma(0)  &=  \exp\!\left(\frac{1}{2}\sin(2\pi x)\sin(2\pi y)\right)\!,  \\
    \bfmu(0)   &=  \bfzero,  \\
    \zeta(0)   &=  \left(1 + \frac{1}{C_V}\right)\sin(2\pi x)\sin(2\pi y),
\end{align}
\end{subequations}
up to projection.
The approximation space $\bbU$ is defined over a uniform grid of triangular cells of width $2^{-5}$; we take a uniform timestep $\Delta t = 2^{-7}$.

With $\Re = \infty$, entropy should be conserved both in an exact solution, and in our conservative scheme \eqref{eq:compns_avcpg}.
Fig.~\ref{fig:inviscid_invariants} shows the error in the entropy for each simulation.
The lines terminate when the nonlinear solver fails to converge, potentially due to a solution to the scheme no longer existing;
we observe that \added{our proposed} \deleted{the SP} scheme fails after 515 timesteps, whereas the implicit midpoint scheme fails after 392.
Our scheme \eqref{eq:compns_avcpg} conserves entropy throughout (up to quadrature error, solver tolerances, and machine precision), whereas implicit midpoint does not, introducing spurious (nonphysical) entropy decrease.

\begin{figure}[!ht]
    \centering

    \begin{tikzpicture}
    \begin{axis}[
        xmin = 0,      xmax = 4.5,  xlabel = {time $t$},  xtick distance = 1,
        ymin = 1e-16,  ymax = 1e0,  ymode = log,
        width = \textwidth, height = 0.5\textwidth,
        axis on top,
        % legend pos = north east,
        legend style = {at = {(0.97, 0.85)}, anchor = north east},
    ]

    \input{plots/inviscid_invariants/entropy/andrews_farrell.tex}
    \addlegendentry{our scheme}
    \input{plots/inviscid_invariants/entropy/implicit_midpoint.tex}
    \addlegendentry{implicit midpoint}
    \end{axis}
    \end{tikzpicture}

    \caption{Error in the entropy $|Q_4 - Q_4(0)|$ over time within the inviscid test (Section~\ref{sec:inviscid_test}) for implicit midpoint and our proposed scheme.}

    \label{fig:inviscid_invariants}
\end{figure}

\subsubsection{Supersonic test} \label{sec:supersonic_test}
We next consider a supersonic perturbation in the velocity field, with $\Pr = 0.71$, typical for air, and $\Re = 2^7$.
The initial conditions are
\begin{subequations}
\begin{align}
    \sigma(0)  &=  1,  \\
    \bfmu(0)   &=  2^3 \exp(\cos(2\pi x) + \cos(2\pi(y-0.5)) - 2) \bfe_1,  \\
    \zeta(0)   &=  0,
\end{align}
\end{subequations}
again up to projection.
The approximation space $\bbU$ is defined over a grid of square cells of uniform width $2^{-8}$; we take a uniform timestep $\Delta t = 2^{-11}$.

Fig.~\ref{fig:supersonic_plots} shows plots of the velocity, density, temperature, and specific entropy at various times in the structure-preserving scheme; the results from the implicit midpoint scheme exhibit very little visual difference.
The shockwave is clearly visible at the final time.
We use continuous approximations to all variables, causing oscillations in $\rho$ and $s$; this could be improved with a non-conforming discontinuous Galerkin spatial discretization.
Fig.~\ref{fig:supersonic_invariants} shows the error in the mass, momentum, and energy for each simulation.
Each is conserved (up to quadrature error, solver tolerances and machine precision) for the scheme \eqref{eq:compns_avcpg}, while only the mass is conserved for the implicit midpoint scheme (or any higher-order Gauss method).
The error in the energy increases exponentially in the implicit midpoint scheme from the point of formation of the shockwave, rising from a value of around $4.046$ to around $4.059$.

\begin{figure}[!ht]
    \captionsetup[subfigure]{justification = centering}
    \centering

    \begin{subfigure}{0.04\textwidth}
        \centering

        $\bfu$

        \vspace{17.5mm}
    \end{subfigure}%
    \begin{subfigure}{0.30\textwidth}
        \centering

        \includegraphics[width = 0.9\textwidth]{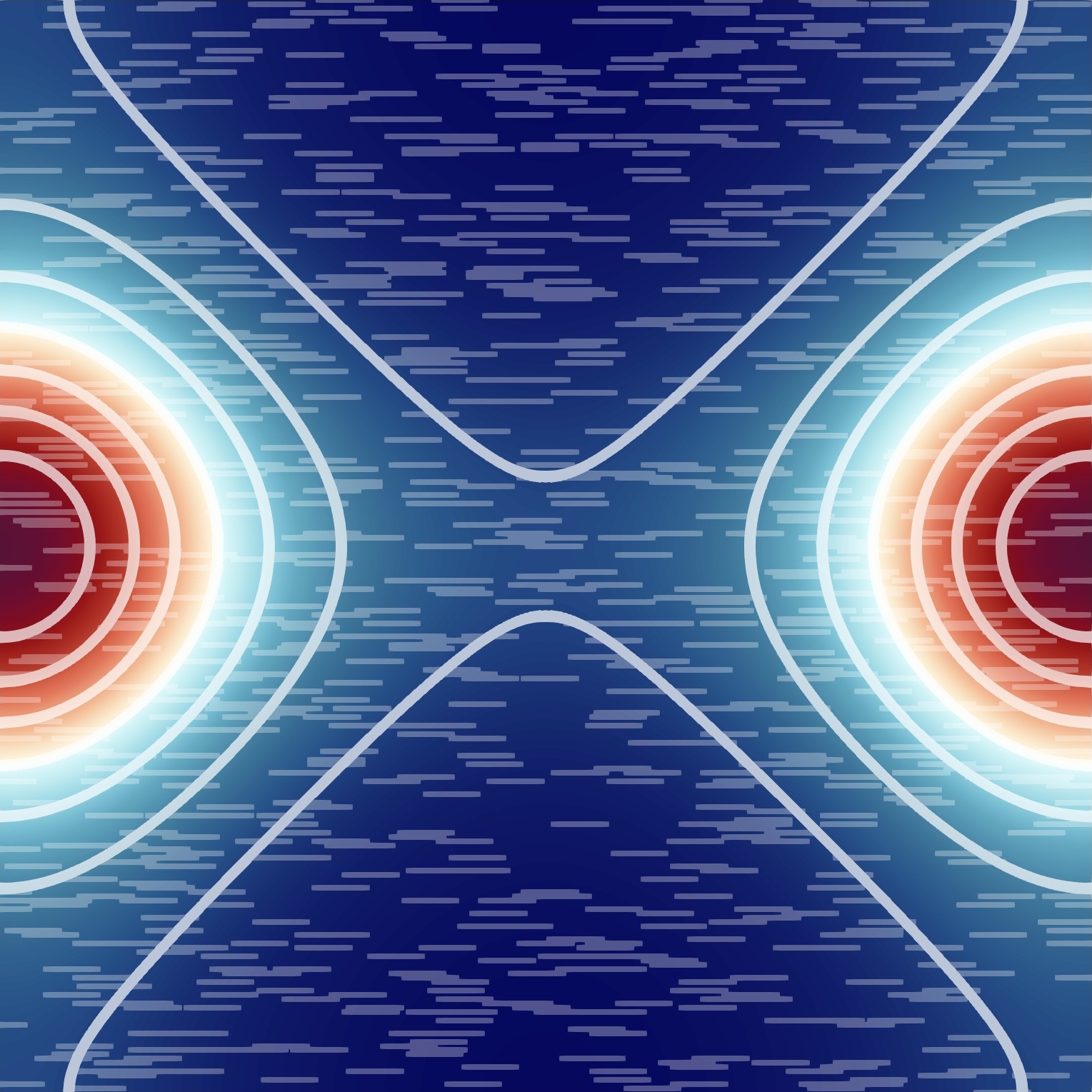}
    \end{subfigure}%
    \begin{subfigure}{0.30\textwidth}
        \centering

        \includegraphics[width = 0.9\textwidth]{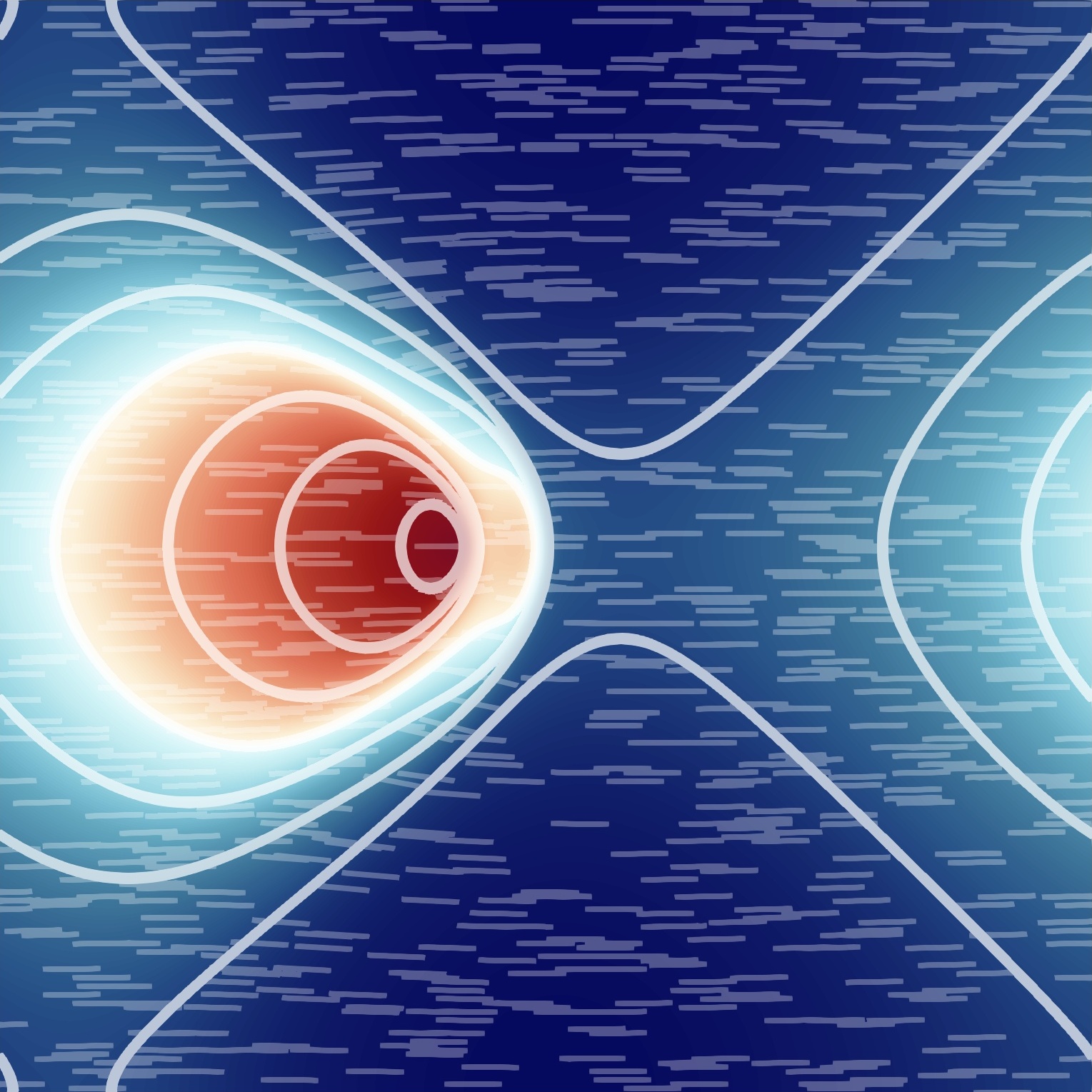}
    \end{subfigure}%
    \begin{subfigure}{0.30\textwidth}
        \centering

        \includegraphics[width = 0.9\textwidth]{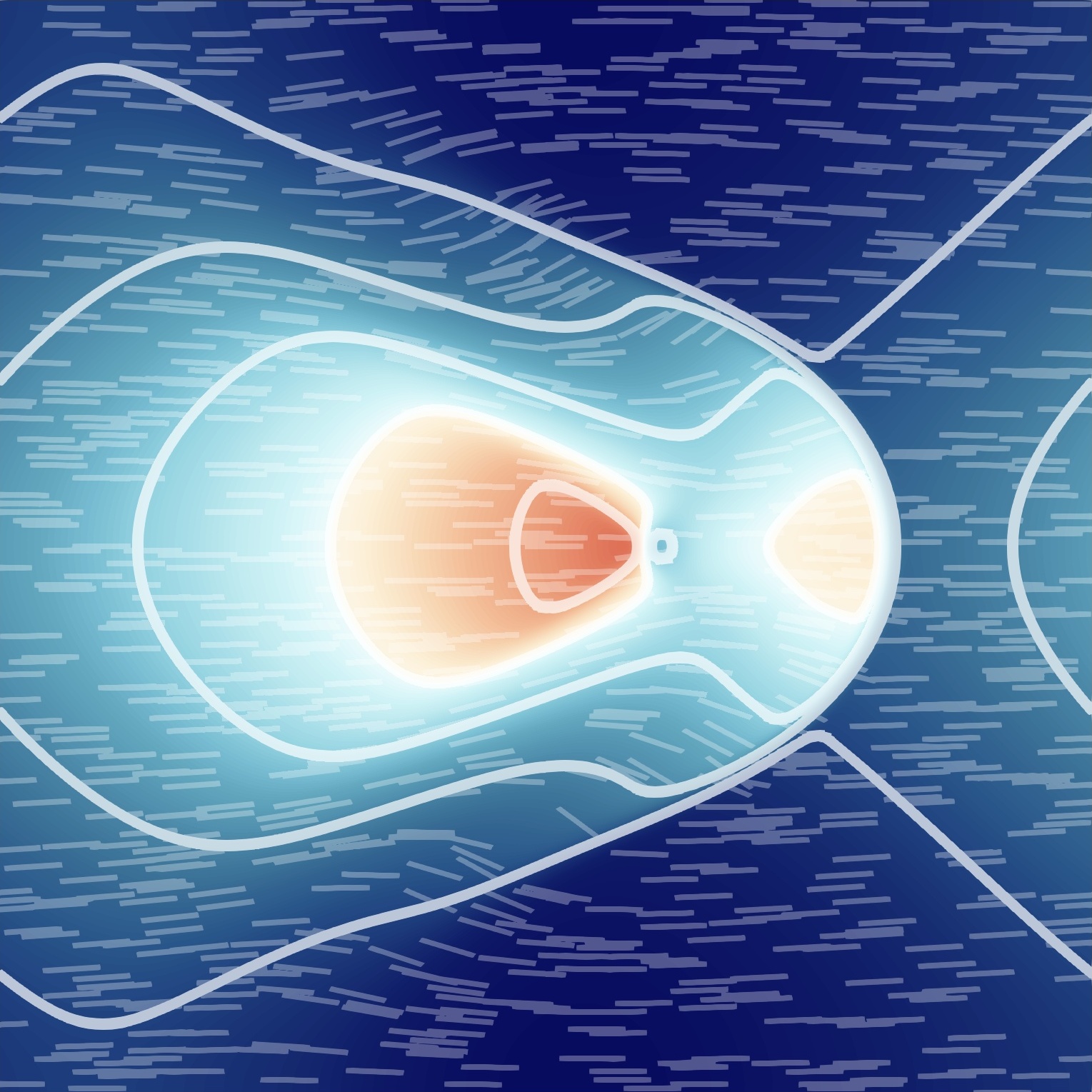}
    \end{subfigure}

    \vspace{3.5mm}

    \begin{subfigure}{0.04\textwidth}
        \centering

        $\rho$

        \raisebox{17.5mm}{}
    \end{subfigure}%
    \begin{subfigure}{0.30\textwidth}
        \centering

        \includegraphics[width = 0.9\textwidth]{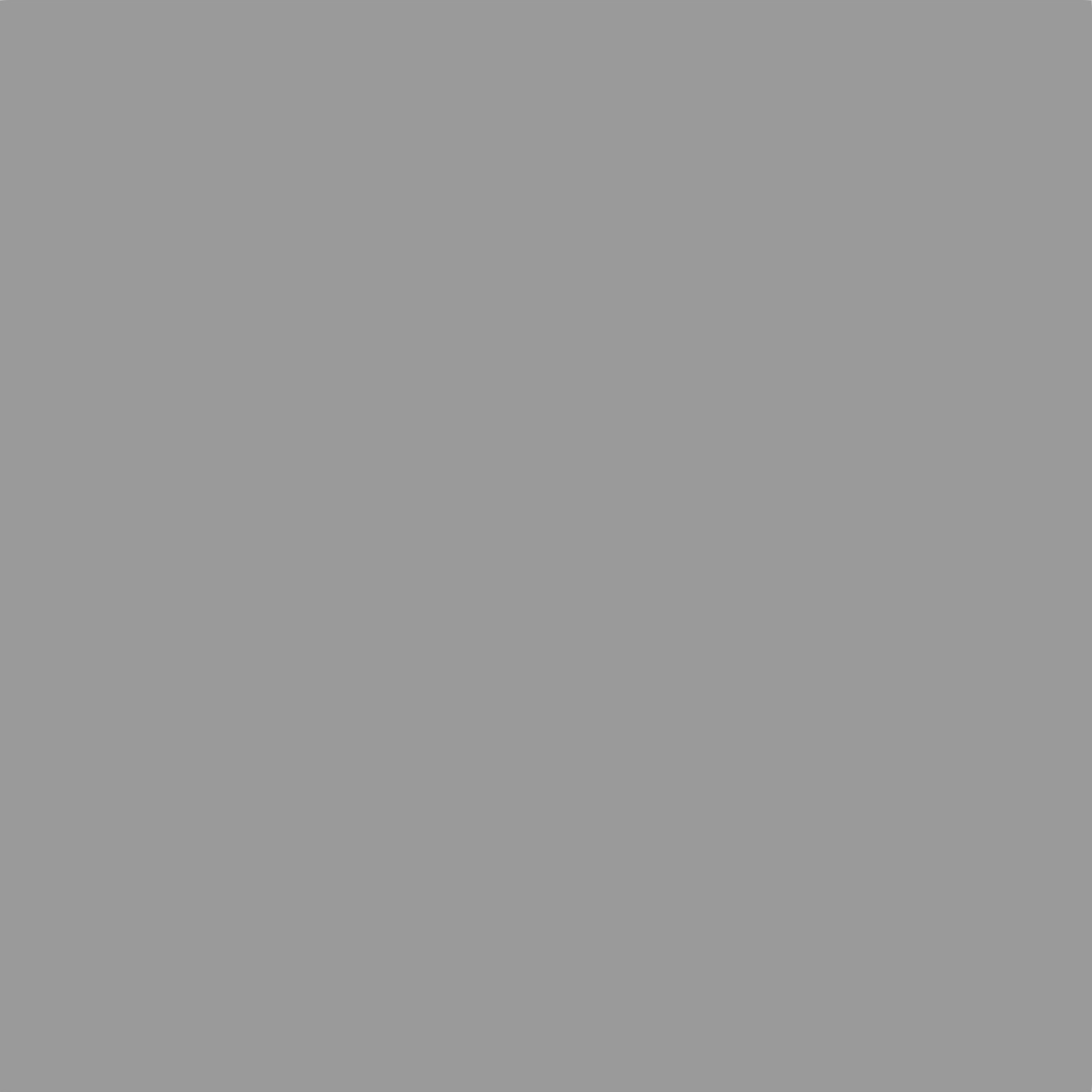}
    \end{subfigure}%
    \begin{subfigure}{0.30\textwidth}
        \centering

        \includegraphics[width = 0.9\textwidth]{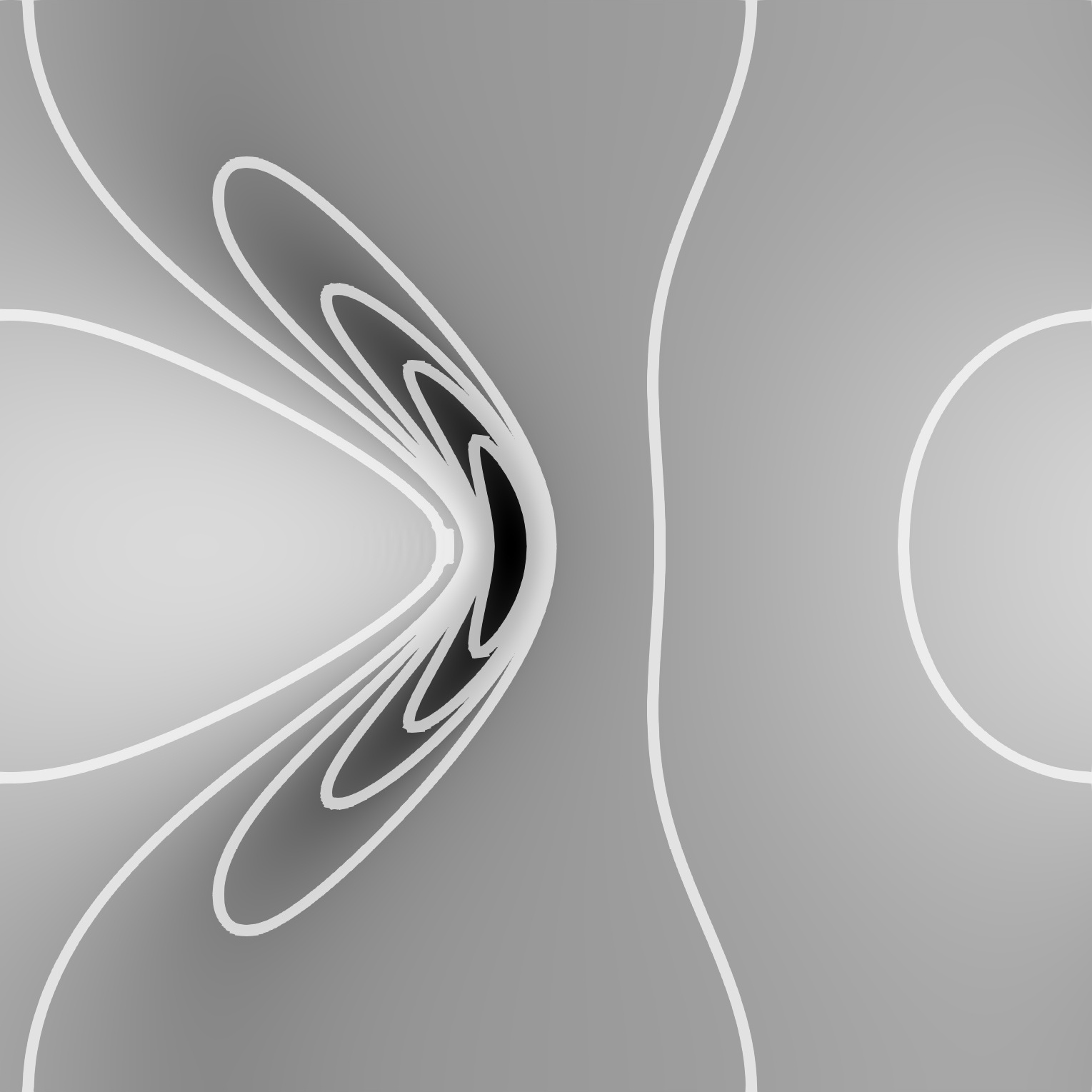}
    \end{subfigure}%
    \begin{subfigure}{0.30\textwidth}
        \centering

        \includegraphics[width = 0.9\textwidth]{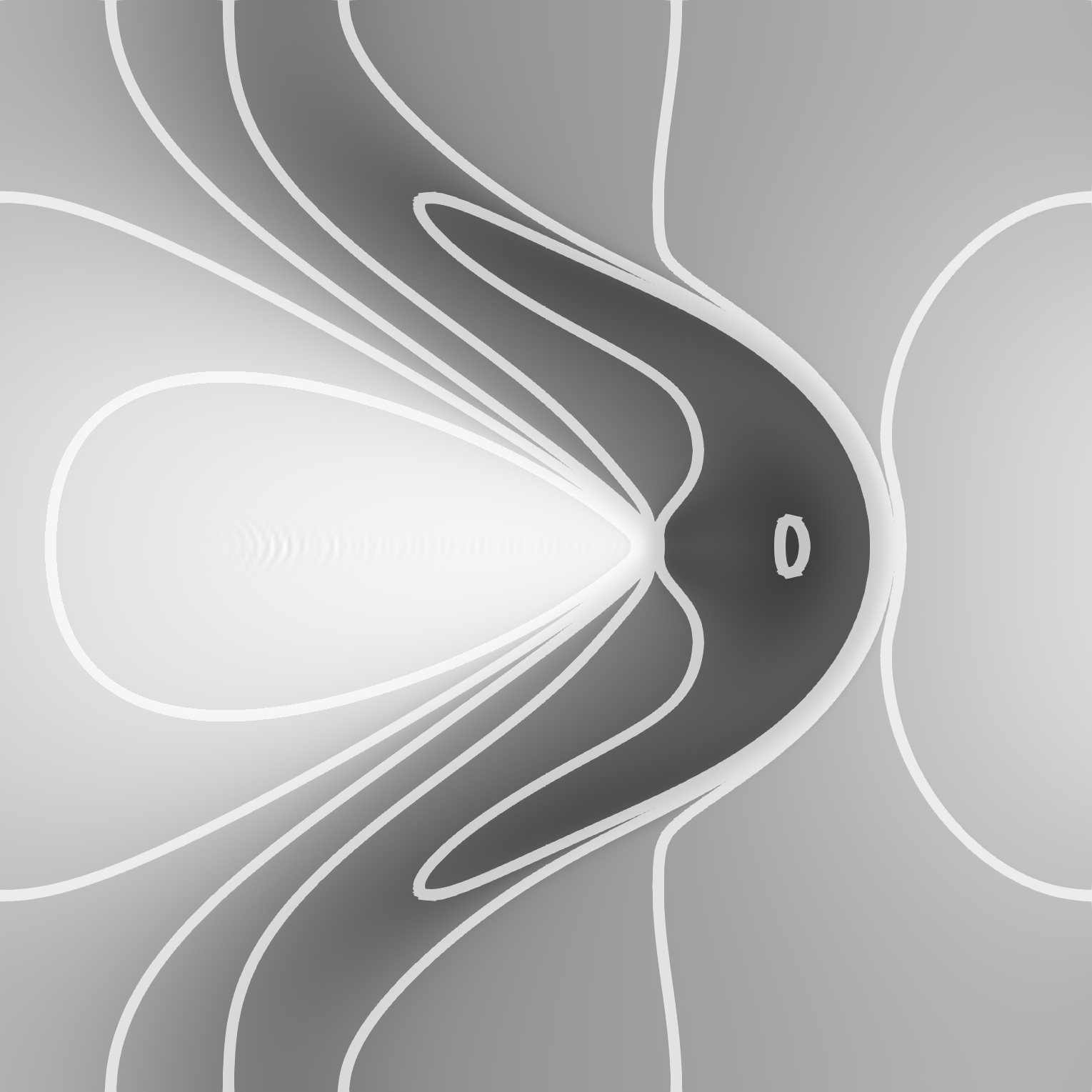}
    \end{subfigure}

    \vspace{3.5mm}

    \begin{subfigure}{0.04\textwidth}
        \centering

        $\theta$

        \raisebox{17.5mm}{}
    \end{subfigure}%
    \begin{subfigure}{0.30\textwidth}
        \centering

        \includegraphics[width = 0.9\textwidth]{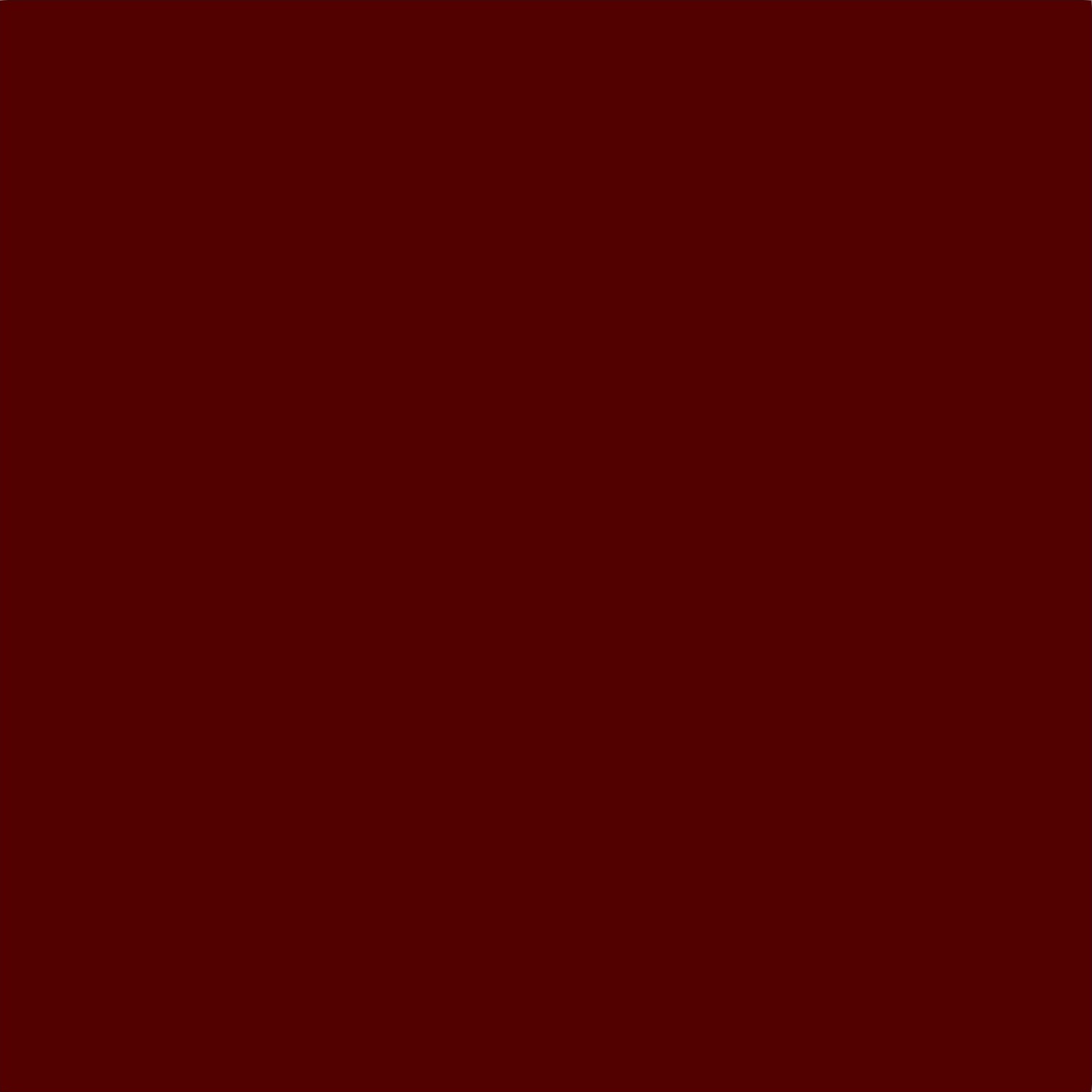}
    \end{subfigure}%
    \begin{subfigure}{0.30\textwidth}
        \centering

        \includegraphics[width = 0.9\textwidth]{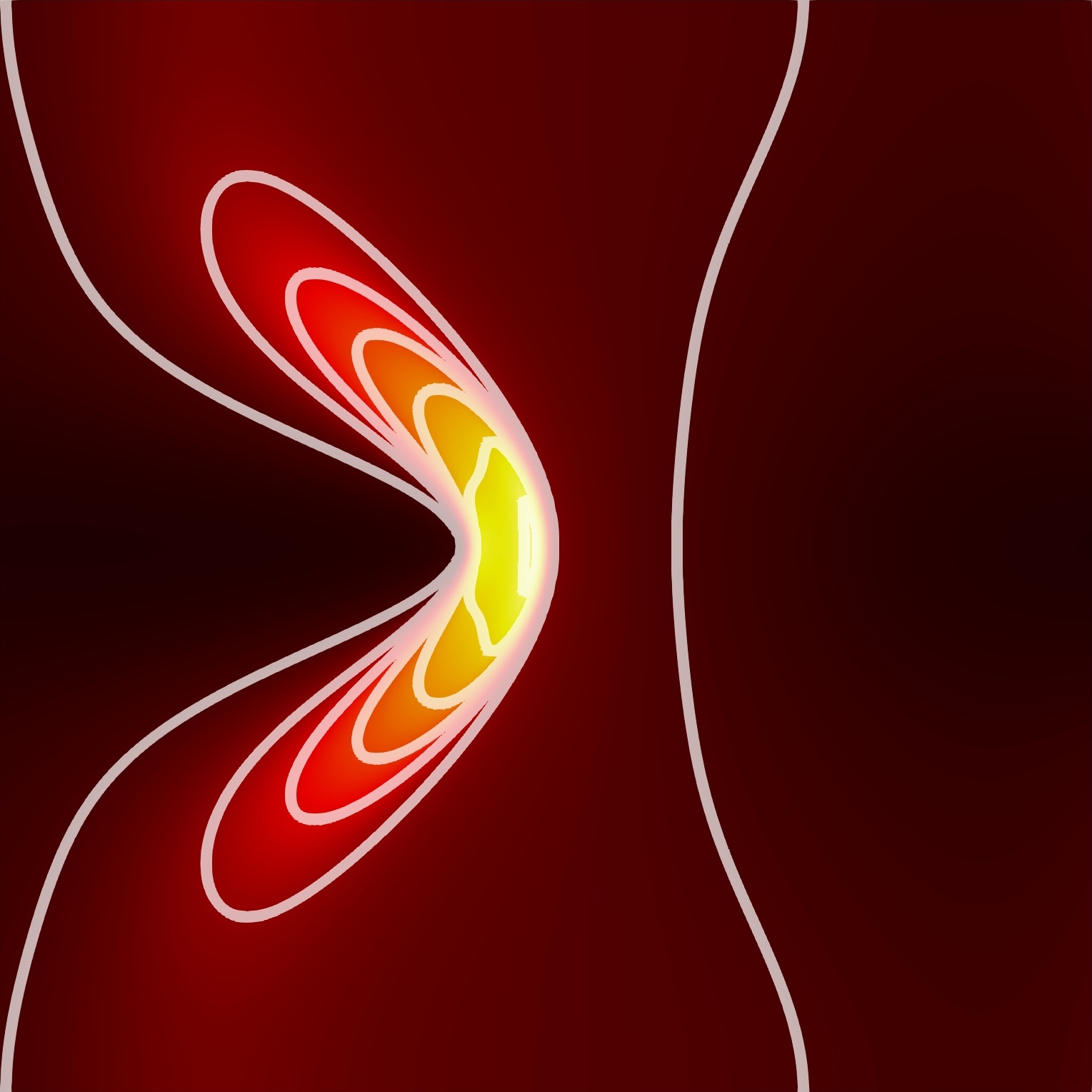}
    \end{subfigure}%
    \begin{subfigure}{0.30\textwidth}
        \centering

        \includegraphics[width = 0.9\textwidth]{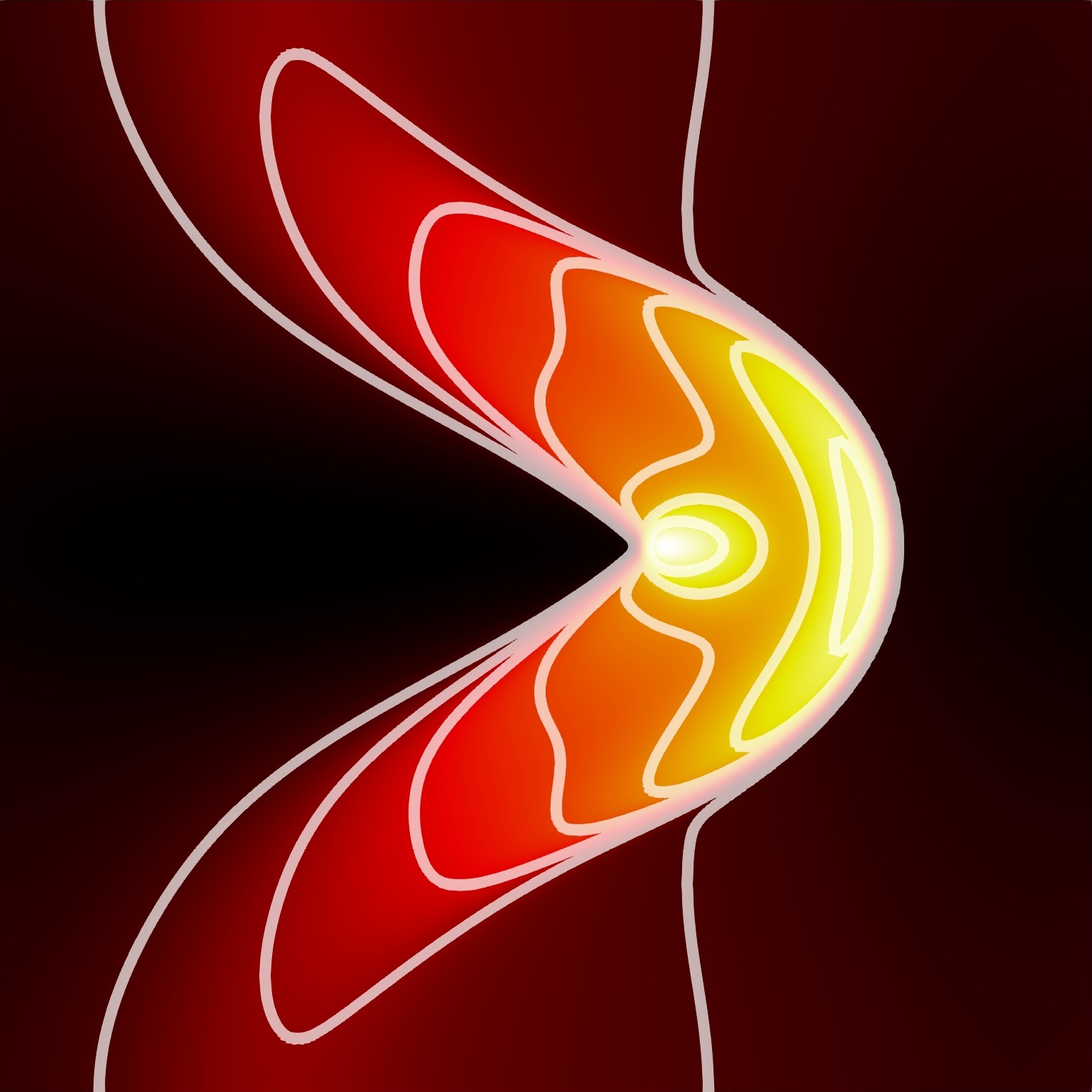}
    \end{subfigure}

    \vspace{3.5mm}

    \begin{subfigure}{0.04\textwidth}
        \centering

        $s$

        \raisebox{17.5mm}{}
    \end{subfigure}%
    \begin{subfigure}{0.30\textwidth}
        \centering

        \includegraphics[width = 0.9\textwidth]{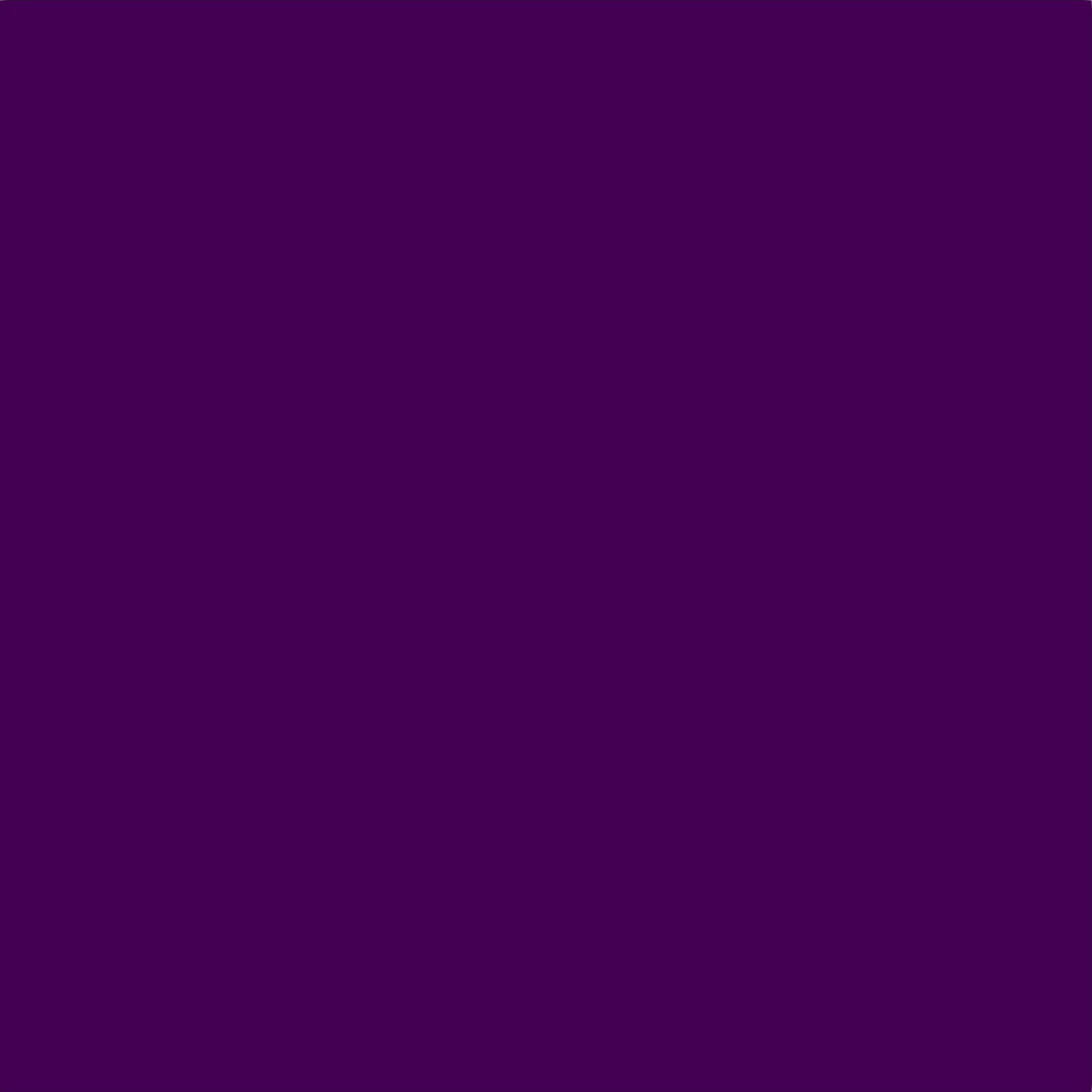}
    \end{subfigure}%
    \begin{subfigure}{0.30\textwidth}
        \centering

        \includegraphics[width = 0.9\textwidth]{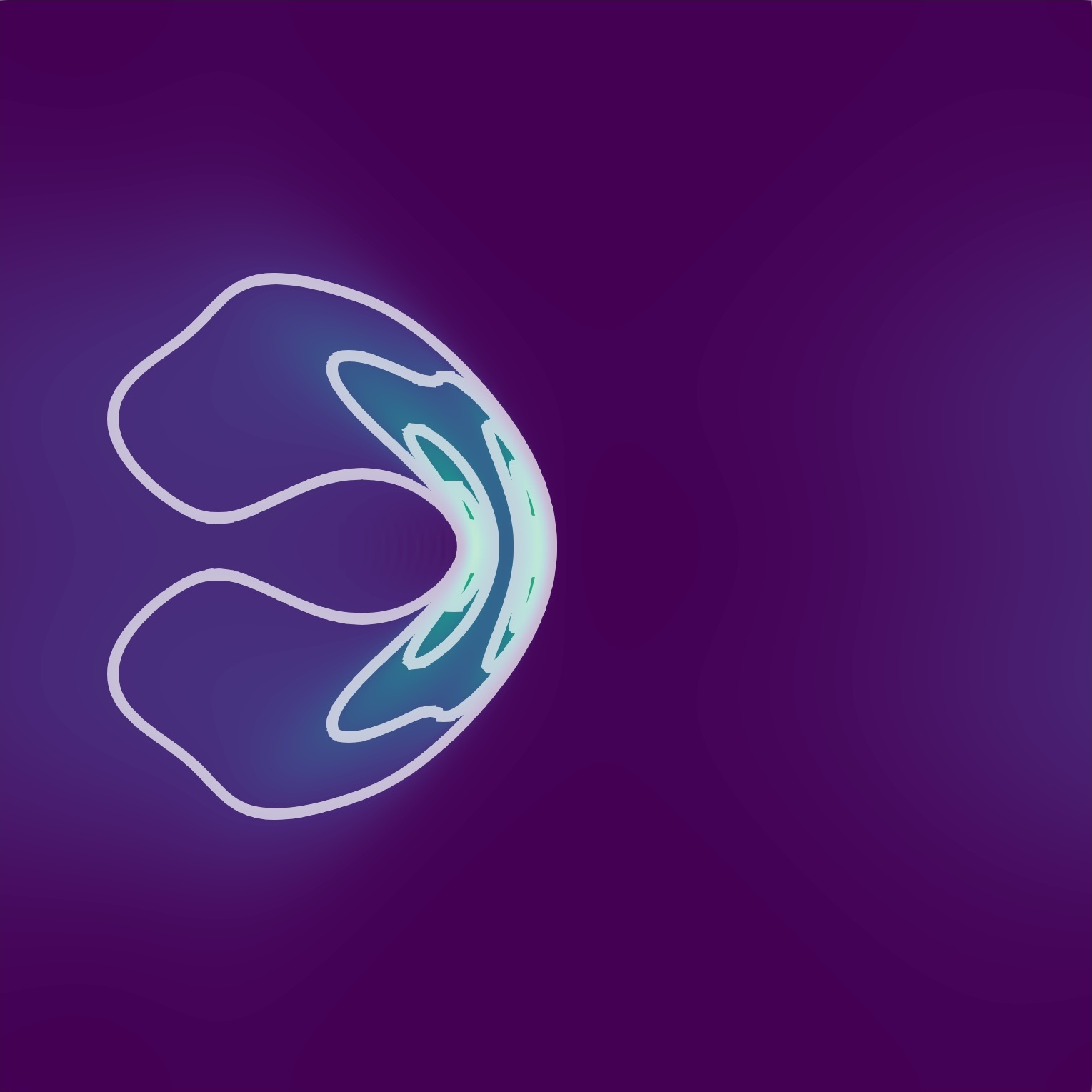}
    \end{subfigure}%
    \begin{subfigure}{0.30\textwidth}
        \centering

        \includegraphics[width = 0.9\textwidth]{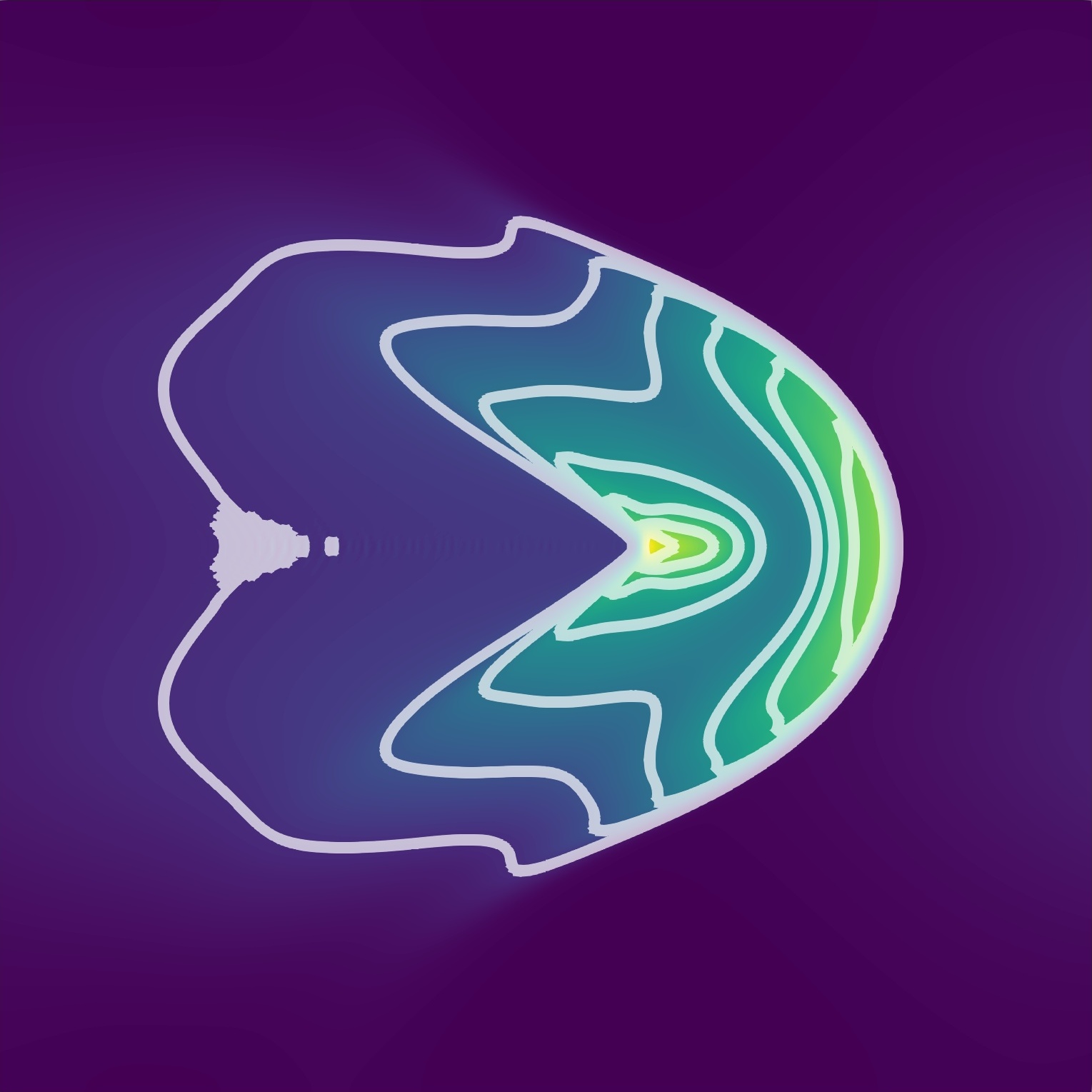}
    \end{subfigure}

    \vspace{2mm}

    \begin{subfigure}{0.04\textwidth}
        $\,$
    \end{subfigure}%
    \begin{subfigure}{0.30\textwidth}
        \centering

        $t = 0$
    \end{subfigure}%
    \begin{subfigure}{0.30\textwidth}
        \centering

        $t = 1 \cdot 2^{-4}$
    \end{subfigure}%
    \begin{subfigure}{0.30\textwidth}
        \centering

        $t = 2 \cdot 2^{-4}$
    \end{subfigure}

    \caption{Contours of the velocity magnitude $\|\bfu\|$, density $\rho$, temperature $\theta$, and specific entropy $s$ at times $t \in \{0, 1 \cdot 2^{-4}, 2 \cdot 2^{-4}\}$ in the supersonic test (Section~\ref{sec:supersonic_test}).}

    \label{fig:supersonic_plots}
\end{figure}

\begin{figure}[!ht]
    \captionsetup[subfigure]{justification = centering}
    \centering

    \begin{subfigure}{0.333\textwidth}
        \centering

        \begin{tikzpicture}
        \begin{axis}[
            xmin = 0,      xmax = 0.25,  xlabel = {$t$},  xtick distance = 0.1,
            ymin = 1e-16,  ymax = 1e0,   ymode = log,
            width = \textwidth, height = 1.5\textwidth,
            axis on top,
            legend pos=north east,
            legend style={font=\small},
        ]
        \input{plots/supersonic_invariants/mass/avcpg.tex}
        \addlegendentry{our scheme}
        \input{plots/supersonic_invariants/mass/im.tex}
        \addlegendentry{\shortstack{implicit \\ midpoint}}
        \end{axis}
        \end{tikzpicture}

        \caption{Mass, $|Q_1 - Q_1(0)|$}
    \end{subfigure}%
    \begin{subfigure}{0.333\textwidth}
        \centering

        \begin{tikzpicture}
        \begin{axis}[
            xmin = 0,      xmax = 0.25,  xlabel = {$t$},  xtick distance = 0.1,
            ymin = 1e-16,  ymax = 1e0,   ymode = log,
            width = \textwidth, height = 1.5\textwidth,
            axis on top,
        ]
        \input{plots/supersonic_invariants/momentum/avcpg.tex}
        \input{plots/supersonic_invariants/momentum/im.tex}
        \end{axis}
        \end{tikzpicture}

        \caption{Momentum, $\|\bfQ_2 - \bfQ_2(0)\|$}
    \end{subfigure}%
    \begin{subfigure}{0.333\textwidth}
        \centering

        \begin{tikzpicture}
        \begin{axis}[
            xmin = 0,      xmax = 0.25,  xlabel = {$t$},  xtick distance = 0.1,
            ymin = 1e-16,  ymax = 1e0,   ymode = log,
            width = \textwidth, height = 1.5\textwidth,
            axis on top,
        ]
        \input{plots/supersonic_invariants/energy/avcpg.tex}
        \input{plots/supersonic_invariants/energy/im.tex}
        \end{axis}
        \end{tikzpicture}

        \caption{Energy, $|Q_3 - Q_3(0)|$}
    \end{subfigure}

    \caption{Errors in different invariants over time within the supersonic test (Section~\ref{sec:supersonic_test}) for implicit midpoint and our proposed scheme.}

    \label{fig:supersonic_invariants}
\end{figure}

\section{Related literature} \label{sec:related_literature}

With the framework presented, we now discuss its relationship with other works, with a focus on structure-preserving discretization of the Navier--Stokes equations.

In 1994, Simo \& Armero \cite{Simo_Armero_1994} proposed a general form for 1-stage energy-preserving finite element integrators for the incompressible Navier--Stokes equations.
This was extended to the preservation of helicity by Rebholz \cite{Rebholz_2007};
the energy- and helicity-preserving scheme \eqref{eq:incompnsavcpg} derived in Section~\ref{sec:incompressible} can be viewed as an extension of this to higher order in time.
This idea has been generalized to incompressible magnetohydrodynamics by Hu, Lee \& Xu \cite{Hu_Lee_Xu_2021}; see also \cite{Laakmann_Hu_Farrell_2023, He_et_al_2025}.
Again, our framework reproduces these schemes when applied to these problems.

To the best of our knowledge, the \mbox{mass-,} \mbox{momentum-,} energy-conserving and entropy-preserving scheme for the compressible Navier--Stokes equations \eqref{eq:compns_avcpg} presented in Section~\ref{sec:compressible} is novel.
However, structure-preserving methods for the compressible Navier--Stokes equations have been well studied, in particular in the context of finite-volume methods \cite{Feireisl_LukacovaMedvidova_Mizerova_2020}.
The concept of entropy-stable methods was introduced and analyzed by Tadmor for the barotropic Euler equations \cite{Tadmor_1987, Tadmor_2003, Tadmor_2016}.
In the context of discontinuous Galerkin methods, auxiliary variables mirroring ours for entropy preservation were introduced by Parsani et al.~and Chan \cite{Parsani_et_al_2016, Chan_2018}, preserving the generation of entropy in the semi-discrete case, discretized in space only;
see also \cite{Chan_2020, Chan_2025}.
See Chen \& Shu \cite{Chen_Shu_2020} for a review on different types of entropy-stable schemes for the compressible Navier--Stokes equations.

The root-density variable was employed by Morinishi and Halpern \& Waltz \cite{Morinishi_2010, Halpern_Waltz_2018}; see also Nordstr\"om \cite{Nordstrom_2022}, where similar forms are used in the context of entropy generation.
Kennedy \& Gruber \cite{Kennedy_Gruber_2008} used a similar decomposition of the convective term to that of \eqref{eq:compns2b} to improve the skew-symmetry of their finite difference discretization and hence improve its energy conservation properties.
The problem of constructing energy-preserving finite element schemes for compressible flow was considered in the ideal limit by Gawlik \& Gay-Balmaz \cite{Gawlik_Gay-Balmaz_2021a}, building on a Lagrangian interpretation of the Navier--Stokes equations \cite{Pavlov_et_al_2011};
dissipative terms were later introduced by the same authors \cite{Gawlik_Gay-Balmaz_2021b}.

\added{Radau-IIA methods \cite[Sec.~IV.5]{Hairer_Wanner_2010} are known to have appealing dissipative properties for many systems and functionals.
At lowest order, this class of L-stable methods includes implicit or backward Euler, as studied by e.g.~Shu {et al.}~\cite{Shu_et_al_2025} for reaction--diffusion systems.
However, the discrete form of the preserved dissipation inequality typically differs from the continuous form, through some additional artificial dissipative term.
For instance, for the incompressible Navier--Stokes equations \eqref{eq:strongincompns} as considered in Section~\ref{sec:incompressible}, an implicit Euler time discretization of the semi-discrete form \eqref{eq:incompnsweak_lagrange} dissipates energy $Q_1$ over each timestep $T_n$ according to
\begin{subequations}
\begin{equation}
    Q_1(\bfu(t_{n+1})) - Q_1(\bfu(t_n))  =  - \frac{\Delta t_n}{\Re}\int\|\nabla\bfu(t_{n+1})\|^2 - \frac{1}{2}\int\|\bfu(t_{n+1}) - \bfu(t_n)\|^2,
\end{equation}
our scheme \eqref{eq:incompnsavcpg} on the other hand, at lowest order ($S=1$), exhibits energy dissipation
\begin{equation}
    Q_1(\bfu(t_{n+1})) - Q_1(\bfu(t_n))  =  - \frac{\Delta t_n}{\Re}\int\!\left\|\frac{1}{2}\nabla[\bfu(t_{n+1}) + \bfu(t_n)]\right\|^2\!.
\end{equation}
\end{subequations}
When applied to certain conservative discretizations of the the inviscid compressible Navier--Stokes equations (i.e.~the compressible Euler equations) as considered in Section~\ref{sec:compressible}, Tadmor \cite{Tadmor_2003} similarly observed that implicit Euler methods introduce additional artificial dissipation in the entropy.
While these additional discrete dissipations can aid stability on large timesteps, their presence implies such integrators do not accurately reproduce dissipation inequalities in the sense we desire in this work.
For both the incompressible and compressible Navier--Stokes equations, for example, one would still observe discrete dissipation in the inviscid limits $\Re = \infty$.}

\section{Conclusions}\label{sec:conclusions}
In this work we have proposed a framework for devising time discretizations that preserve conservation laws and dissipation inequalities.
The framework represents quantities of interest by means of their associated test functions, and introduces auxiliary variables for each.
Our approach generalizes existing works in the literature, and offers new structure-preserving discretizations for challenging problems such as the compressible Navier--Stokes equations. We have also successfully applied these ideas to deriving (multi-)conservative discretizations of Hamiltonian ordinary and partial differential equations.

A key challenge in rendering the proposed time discretizations competitive is that the nonlinear systems at each timestep are large and coupled, with possibly many auxiliary variables, and possibly many copies of each variable at high-order in time. This latter property is shared by fully implicit Runge--Kutta discretizations, and solvers for these systems have received increasing attention in recent years \cite{Southworth_2022, Axelsson_2023, Munch_2023}.
We hope that the clever solver approaches and implementation for such schemes can be extended to the systems arising from our framework.

\section{Code availability}
The code used for the numerical results employed Firedrake \cite{Ham_et_al_2023}, PETSc \cite{Balay_et_al_2024}, and MUMPS \cite{Amestoy_2001}.
All code for reproducing the numerical results of this work and major Firedrake components have been archived at \cite{Andrews_Farrell_2025_Zenodo}.

\section*{Acknowledgments}
%We acknowledge many helpful discussions with 
We would like to acknowledge the kind support of W.~Arter and R.~Akers of the UK Atomic Energy Authority.

\appendix

\bibliographystyle{siamplain}
\bibliography{references}
\end{document}